\documentclass[12pt]{amsart}
\usepackage[latin1]{inputenc}
\usepackage{color}
\usepackage{comment}
\usepackage{dsfont}
\usepackage{enumerate}
\usepackage{amssymb}
\usepackage[all]{xy}
\usepackage{hyperref}

\newtheorem{thm}{Theorem}[section]
\newtheorem{cor}[thm]{Corollary}
\newtheorem{lem}[thm]{Lemma}

\newtheorem{prop}[thm]{Proposition}
\theoremstyle{definition}
\newtheorem{defn}[thm]{Definition}
\newtheorem{question}[thm]{Question}
\theoremstyle{remark}
\newtheorem{rem}[thm]{Remark}

\newtheorem{example}[thm]{Example}
\numberwithin{equation}{section}

\newcommand{\fbl}{\text{FBL}}
\newcommand{\R}{\mathbb{R}}
\newcommand{\N}{\mathbb{N}}

\DeclareMathOperator{\lat}{lat}

\DeclareMathOperator{\Hom}{Hom}
\DeclareMathOperator{\FBL}{FBL}

\usepackage{color}
\usepackage{enumerate}
\begin{document}
\setcounter{tocdepth}{1}

\title[Norm-attaining lattice homomorphisms]{Norm-attaining lattice homomorphisms and renormings of Banach lattices}

\author[Bilokopytov]{Eugene Bilokopytov}
\address{Department of Mathematical and Statistical Sciences, University of Alberta, Edmonton, Canada.
\newline
	\href{https://orcid.org/0000-0001-7075-886X}{ORCID: \texttt{0000-0001-7075-886X} } }
\email{bilokopy@ualberta.ca}

\author[Garc\'ia-S\'anchez]{Enrique Garc\'ia-S\'anchez}
\address{Instituto de Ciencias Matem\'aticas (CSIC-UAM-UC3M-UCM)\\
Consejo Superior de Investigaciones Cient\'ificas\\
C/ Nicol\'as Cabrera, 13--15, Campus de Cantoblanco UAM\\
28049 Madrid, Spain.
\newline
	\href{https://orcid.org/0009-0000-0701-3363}{ORCID: \texttt{0009-0000-0701-3363} } }
\email{enrique.garcia@icmat.es}

\author[de Hevia]{David de Hevia}
\address{Instituto de Ciencias Matem\'aticas (CSIC-UAM-UC3M-UCM)\\
Consejo Superior de Investigaciones Cient\'ificas\\
C/ Nicol\'as Cabrera, 13--15, Campus de Cantoblanco UAM\\
28049 Madrid, Spain.}
\email{david.dehevia@icmat.es}

\author[Mart\'inez-Cervantes]{Gonzalo Mart\'inez-Cervantes}
\address[]{Universidad de Murcia, Departamento de Matem\'{a}ticas, Campus de Espinardo 30100 Murcia, Spain.
	\newline
	\href{http://orcid.org/0000-0002-5927-5215}{ORCID: \texttt{0000-0002-5927-5215} } }	
\email{gonzalo.martinez2@um.es}

\author[Tradacete]{Pedro Tradacete}
\address{Instituto de Ciencias Matem\'aticas (CSIC-UAM-UC3M-UCM)\\
Consejo Superior de Investigaciones Cient\'ificas\\
C/ Nicol\'as Cabrera, 13--15, Campus de Cantoblanco UAM\\
28049 Madrid, Spain.
	\newline
	\href{https://orcid.org/0000-0001-7759-3068}{ORCID: \texttt{0000-0001-7759-3068} }}
\email{pedro.tradacete@icmat.es}



\begin{abstract}
A well-known theorem due to R. C. James states that a Banach space is reflexive if and only if every bounded linear functional attains its norm. In this note we study Banach lattices on which every (real-valued) lattice homomorphism attains its norm. Contrary to what happens in the Banach space setting, we show that this property is not invariant under lattice isomorphisms.
Namely, we show that in an AM-space every lattice homomorphism attains its norm, whereas every infinite-dimensional $C(K)$ space admits an equivalent lattice norm with a lattice homomorphism which does not attain its norm.
Furthermore, we characterize coordinate functionals of atoms and show that whenever a Banach lattice $X$ supports a strictly positive functional, there exists a renorming with the property that the only (non-trivial) lattice homomorphisms attaining their norm are precisely these coordinate functionals. 
As a consequence, one can exhibit examples of Dedekind complete Banach lattices admitting a renorming with a non-norm-attaining lattice homomorphism, answering negatively questions posed by Dantas, Rodr\'iguez Abell\'an, Rueda Zoca and the fourth author.
\end{abstract}

\maketitle

\section{Introduction}

The study of norm attainment for linear functionals and operators is a fundamental aspect of functional analysis, with profound implications in the theory of Banach spaces. One of the cornerstone results in this area is James' Theorem \cite{J}, which provides a characterization of reflexive Banach spaces. Namely, it asserts that a Banach space is reflexive if and only if every bounded linear functional attains its norm, where a functional $x^* \in X^*$ in a dual Banach space is said to attain its norm if there exists $x\in B_X$ such that $x^*(x)=\|x^*\|$. This theorem has served as a pivotal point of reference in understanding the geometry of Banach spaces and the behavior of linear functionals.\smallskip
	
In the context of Banach lattices, the situation becomes more intricate. Banach lattices, which are Banach spaces equipped with a lattice structure, provide a natural setting where the questions about norm attainment for positive functionals and lattice homomorphisms open up new avenues of exploration that extend beyond the traditional Banach space framework. \smallskip
	
Lattice homomorphisms are precisely the operators between Banach lattices that commute with the lattice operations, and are fundamental to understand this theory. Classical developments on lattice homomorphisms include duality relations, extension properties, spectra... We refer the reader to \cite{Meyer} as a reference for these results. More recently, lattice homomorphisms have been the focus of research in the development of the theory of free, injective and projective Banach lattices (see \cite{AMR20,AMRR22b,ART18,AT,GLTT,OTTT,dPW15}). \smallskip

In recent years, there has been an increasing interest in studying norm attainment in the context of positive functionals, as shown in \cite{ASM}, \cite{LV} and \cite{OT} for example. In this paper, we will focus on investigating such a problem for the particular case of lattice homomorphisms, continuing the work initiated in \cite{DMCRARZ}. \smallskip

First, we begin by studying a special class of lattice homomorphisms, namely the coordinate functionals of atoms. Throughout Section \ref{section:Coordinate functionals on a Banach lattice} we provide several characterizations of these lattice homomorphisms and we show that every lattice homomorphism on an order-continuous Banach lattice belongs to this family (Corollary \ref{cor:coordinate}). Coordinate functionals have an interesting property: not only do they attain their norm, but also they are norm-attaining for any lattice renorming of the Banach lattice. \smallskip

This leads to the following question (which is the motivation for Section \ref{section:Stability of norm attainment of lattice homomorphism under renormings}): if a lattice homomorphism $x^*$ on a Banach lattice $X$ is not a coordinate functional, is there a lattice renorming of $X$ for which $x^*$ does not attain its norm? It turns out that in many cases this is indeed true: if a Banach lattice $X$ has a strictly positive functional, then $X$ may be renormed in such a way that the only lattice homomorphisms which attain their norm are coordinate functionals of atoms (Theorem \ref{prop:strictly-positive-functional}). This theorem tells us how we can renorm $\ell_\infty$ (or $L_\infty[0,1]$) in order to obtain a Dedekind complete Banach lattice with lattice homomorphisms which do not attain their norm. This answers in the negative a question posed in \cite{DMCRARZ}. Returning to the question we asked at the beginning of the paragraph, we also provide an example of a lattice homomorphism which does attain its norm for any lattice renorming despite not being a coordinate functional (Example \ref{ex:omega1}). What is true in general is that if a lattice homomorphism attains its norm always at the same (suitably normalized) vector, then it must be a coordinate functional (Corollary \ref{cor:converse-coordinate functionals}). \smallskip

Next, we devote Sections \ref{sec:AMspaces} and \ref{section:renormings AMspaces} to understanding norm-attaining lattice homomorphisms in AM-spaces. Recall that by Kakutani's representation theorem \cite{Kakutani-AM spaces}, every AM-space may be identified (in a lattice isometric way) with a sublattice of a certain $C(K)$-space. Thus, it is interesting to know what properties of $C(K)$-spaces can be extended to AM-spaces. A first similarity to be noted is that AM-spaces do have a \textit{large} set of lattice homomorphisms; in fact, AM-spaces can be described as those Banach lattices whose set of lattice homomorphisms is $1$-norming (Proposition \ref{prop:norming}). It is clear that in any $C(K)$-space every lattice homomorphism (which must be a positive multiple of the evaluation at a point in $K$) attains its norm at the constant function $\mathbf{1}_K$. As AM-spaces do not necessarily possess a strong unit, it is not a priori obvious that their lattice homomorphisms must attain their norm. Despite this, we prove in Theorem \ref{THMAM} that every lattice homomorphism on an AM-space does attain its norm. We deduce this result as an easy consequence of a characterization of norm-attaining lattice homomorphisms which is relevant in its own right (Proposition \ref{prop:characterizationLatNA}). Regarding lattice renormings, it is a simple observation that every infinite-dimensional $C(K)$-space admits an equivalent lattice norm with a non-norm-attaining lattice homomorphism (Proposition \ref{PropNotNA}); in fact, this is the case for all AM-spaces aside of the spaces of the form $c_0(\Gamma)$  (Theorem \ref{thm:several}). \smallskip

Sections \ref{section:Free Banach lattices over Banach spaces} and \ref{Section:FBLlattices} are devoted to another class of Banach lattices with large sets of lattice homomorphisms: free Banach lattices. In Section \ref{section:Free Banach lattices over Banach spaces} we show that free $p$-convex Banach lattices generated by Banach spaces also admit renormings with non-norm-attaining lattice homomorphism (Corollary \ref{cor:FBL}), and we revisit the problem of characterizing the norm-attaining lattice homomorphisms of $\fbl^{(p)}[E]$ established in \cite[Conjecture 5.5]{DMCRARZ} (see Question \ref{QuestionFBL}). Finally, motivated by this discussion, we consider in Section \ref{Section:FBLlattices} free $p$-convex Banach lattices over lattices, $\fbl^{(p)}\langle\mathbb{L}\rangle$, introduced in \cite{AR-A19}. In contrast to what happens for free Banach lattices over Banach spaces, every lattice homomorphism on $\fbl^{(p)}\langle\mathbb{L}\rangle$ (for any lattice $\mathbb{L}$) is norm-attaining (Proposition \ref{prop:normattainingFBLlattice}). We obtain this last result as a consequence of the aforementioned characterization of norm-attaining lattice homomorphisms (Proposition \ref{prop:characterizationLatNA}). \smallskip

\section{Coordinate functionals on a Banach lattice}\label{section:Coordinate functionals on a Banach lattice}
	
A \textit{Banach lattice} is a Banach space $X$ equipped with a partial order $\leq$ such that  $(X, \leq)$ is a vector lattice and the norm $\| \cdot \|$ on $X$ is a lattice norm, i.e., if $|x| \leq |y|$, then $\|x\| \leq \|y\|$. The dual of a Banach lattice is again a Banach lattice, where a (linear) functional $x^* \in X^*$ is \textit{positive} whenever it sends positive elements to positive numbers, i.e.~if $x^*(x)\geq0$ whenever $x\geq 0$.
Positive functionals may not preserve infima and suprema. In fact, those functionals $x^*\in X^*$ satisfying $x^*(x\wedge y)=x^*(x) \wedge x^*(y)$ (or equivalently, $x^*(x\vee y)=x^*(x) \vee x^*(y)$) are called \textit{lattice homomorphisms}. The set of all functionals in $X^*$ that are lattice homomorphisms is denoted by $\Hom(X, \R)$, whereas the set of all functionals in $X^*$ that attain their norm is denoted by $\text{NA}(X,\mathbb{R})$. We refer the reader to \cite{LinTza2, Meyer} for unexplained terminology about Banach lattices and lattice homomorphisms.\medskip

As we mentioned in the introduction, the central topic of the present work are lattice homomorphisms that attain their norm. Later in the paper some specific classes of Banach lattices will be studied, but first we focus our scope on general facts about the stability of the norm attainment of lattice homomorphisms under lattice renormings of a Banach lattice (by \textit{lattice renorming} we mean a Banach lattice norm $|||\cdot|||$ with respect to the same order and lattice operations). It turns out that coordinate functionals of an atom are lattice homomorphisms that always attain the norm, independently of the lattice renorming, as will be shown in the following proposition. \medskip

Recall that the ideal generated by an atom $x_{0}\in X_{+}$ is always a projection band and its corresponding band projection $P_{x_0}:X\to X$ is given by $P_{x_0}(x)=\sup_{n}(x\land nx_0)$, for $x\in X_+$. We say that  $\lambda_{x_0}\in X^{*}$ is the \textit{coordinate functional} of $x_{0}$ if $P_{x_0}x=\lambda_{x_0}\left(x\right)x_{0}$, for every $x\in X$. In particular, in this case $\lambda_{x_0}$ is a lattice homomorphism, $\ker \lambda_{x_0}=x_{0}^{d}$, and $\lambda_{x_0}\left(x_{0}\right)=1$. Note that a positive multiple of a coordinate functional of an atom is also a coordinate functional of an atom (as for $r>0$, we would have $P_{x_0}x=r\lambda_{x_0}\left(x\right)\frac{1}{r} x_{0}$ for $x\in X$, so $r\lambda_{x_0}=\lambda_{\frac{1}{r}x_0}$).

\begin{prop}\label{prop:coordinate}
Let $x^*$ be a non-zero lattice homomorphism on a Banach lattice $X$, and let $x_{0}\in X_{+}$ be an atom. The following conditions are equivalent.
\begin{enumerate}
\item $x^*$ is a (positive) multiple of the coordinate functional of $x_{0}$;
\item For any lattice norm on $X$, $x^*$ attains its norm at a multiple of $x_{0}$;
\item There is a lattice norm on $X$ such that $x^*$ attains its norm at $x_{0}$;
\item $x^*(x_{0})>0$;
\item $x_{0}^{d}\subset \ker x^*$.\medskip
\end{enumerate}
In this case, $y^*\in X^*$ is disjoint with $x^*$ if and only if $y^*(x_0)=0$.
\end{prop}
\begin{proof}
(i)$\Rightarrow$(ii): We may assume that $x^*=\lambda_{x_0}$ (the coordinate functional of $x_{0}$) and let $P_{x_0}$ denote its corresponding band projection. For any norm $\|\cdot\|$ on $X$ and any $x\in X$ we have $\|x\|\ge\|P_{x_0}x\|=\|\lambda_{x_0}\left(x\right)x_{0}\|=\left|\lambda_{x_0}\left(x\right)\right|\|x_{0}\|$. It follows that $$\|\lambda_{x_0}\|\le \frac{1}{\|x_{0}\|}=\frac{\lambda_{x_0}\left(x_{0}\right)}{\|x_{0}\|}= \lambda_{x_0}\left(\frac{x_0}{\|x_0\|}\right)\leq \|\lambda_{x_0}\|,$$ and so $\lambda_{x_0}$ attains its norm at $\frac{x_{0}}{\|x_0\|}$.\medskip

(ii)$\Rightarrow$(iii) is trivial in light of the fact that $X$ is a Banach lattice, so it admits at least one lattice norm. \medskip

(iii)$\Rightarrow$(iv) is trivial.\medskip

(iv)$\Rightarrow$(v): For any $x\bot x_{0}$ we have 
$$
\left|x^*\left(x\right)\right|\wedge x^*\left(x_0\right)=x^*\left(\left|x\right|\wedge x_{0}\right)=x^*\left(0\right)=0,
$$ 
and since $x^*\left(x_0\right)>0$ we conclude that $x\in\ker x^*$.\medskip

(v)$\Rightarrow$(i): Since $x_{0}^{d}$ is of codimension $1$ and $x^*\neq 0$, it follows that $x_{0}^{d}=\ker x^*$. As $x^*$ has the same kernel as the coordinate functional of $x_{0}$, we conclude that the two are multiples of each other.\medskip

Now we assume that $0\leq y^*\bot x^*$. By Riesz-Kantorovich formula, for every $n\in\N$ there is $y_n\in[0,x_0]$ such that $x^*(y_n)\le \frac{1}{n}$ and $y^*(x_0-y_n)\le \frac{1}{n}$. Since $x_0$ is an atom there is $r_n\geq 0$ such that $y_n=r_n x_0$. We then have $r_nx^*(x_0)=x^*(y_n)\le \frac{1}{n}$, hence $r_n\to 0$, and $(1-r_n)y^*(x_0)=y^*(x_0-y_n)\le \frac{1}{n}$. The last inequality yields $y^*(x_0)\le \frac{1}{n(1-r_n)}\to 0$. Thus, $y^*(x_0)=0$ as desired. Without the assumption that $y^*\geq 0$ we still have that $|y^*|\bot x^*$, hence $0=-|y^*|(x_0)\leq y^*(x_0)\leq |y^*|(x_0)=0$.

In order to prove the converse, observe that, by Riesz-Kantorovich, $|y^*|(x_0)=0$, and then using it again we have that $\left[|y^*|\wedge x^*\right](x_0)=0=\left[|y^*|\wedge x^*\right](x)$, for any $x\bot x_0$. It follows that $|y^*|\wedge x^*=0$.
\end{proof}

We now present some criteria for a lattice homomorphism on a vector lattice to be a coordinate functional of an atom.

\begin{prop}\label{prop:coordinate VL}
Let $x^*\ne 0$ be a lattice homomorphism on a vector lattice $X$. The following conditions are equivalent.
\begin{enumerate}
\item There exists an atom $x_0\in X$ such that $x^*(x_0)>0$.
\item $x^*$ is a coordinate functional of an atom in $X$;
\item $\ker x^*$ is a projection band;
\item $\ker x^*$ is not order dense;
\item $x^*$ is order continuous.
\end{enumerate}
\end{prop}
\begin{proof}
(i)$\Rightarrow$(ii) follows from Proposition \ref{prop:coordinate} (the relevant implications (iv)$\Rightarrow$(v)$\Rightarrow$(i) hold for any vector lattice).\medskip 

(ii)$\Rightarrow$(iii)$\Rightarrow$(iv) are straightforward.\medskip

(iv)$\Rightarrow$(i): $\ker x^*$ is an ideal of codimension $1$. If it is not order dense, then $\left(\ker x^*\right)^{d}$ is a non-trivial ideal. It then must have dimension $1$, and so it is the span of an atom, say $x_0$. In particular, $x^*(x_{0})\ne 0$.\medskip

(ii)$\Rightarrow$(v): There is an atom $x_0\in X$ such that $Px=x^*\left(x\right)x_{0}$, for every $x\in X$. Since $P$ is order continuous, it follows that so is $x^*$.\medskip

(v)$\Rightarrow$(iv): It follows that $\ker x^*$ is a proper band, and so cannot be order dense.
\end{proof}

We can add some additional conditions to the characterization above when the domain is a Banach lattice instead of a vector lattice.

\begin{prop}\label{prop:coordinate BL}
Let $x^*\ne 0$ be a lattice homomorphism on a Banach lattice $X$. The following conditions are equivalent.
\begin{enumerate}
\item $x^*$ is a coordinate functional of an atom in $X$;
\item $x^{*d}$ is $w^*$-closed in $X^*$;
\item $x^{*d}\cap B_{X^*}$ is $w^*$-compact;
\item $x^*\notin\overline{x^{*d}}^{w^*}$.
\end{enumerate}
\end{prop}
\begin{proof}
 (i)$\Rightarrow$(ii) follows from Proposition \ref{prop:coordinate}. \medskip
 
 (ii)$\Rightarrow$(iii) and (ii)$\Rightarrow$(iv) are straightforward. \medskip
 
 (iii)$\Rightarrow$(ii) follows from Krein-Smulian theorem.\medskip

(iv)$\Rightarrow$(i): Since $x^{*d}$ is a subspace of $X^*$ it follows that $\overline{x^{*d}}^{w^*}=(x^{*d}_{\bot})^{\bot}$, where $x^{*d}_{\bot}\subset X$ is the set of vectors that vanish on every element of $x^{*d}$. Hence, there is $x_0\in x^{*d}_{~\bot}$ such that $x^*(x_0)>0$. Since $x^{*d}$ is an ideal of $X^*$ of codimension $1$, it follows from the dual Riesz-Kantorovich formula that $x^{*d}_{~\bot}$ is an ideal in $X$ of dimension $1$. Hence, $x_0$ is an atom, and so according to Proposition \ref{prop:coordinate VL} we conclude that $x^*$ is a coordinate functional of an atom.
\end{proof}

Using the last item in Proposition  \ref{prop:coordinate VL} and Proposition \ref{prop:coordinate} we immediately get the following fact, previously established in \cite[Theorem 4.1]{DMCRARZ}.

\begin{cor}\label{cor:coordinate}
Every lattice homomorphism on an order continuous Banach lattice is a coordinate functional of an atom, and in particular, attains its norm.
\end{cor}

The last statement can be generalized to the case of finite-rank operators thanks to the additional information provided by Corollary \ref{cor:coordinate}:

\begin{cor}\label{cor:finrankoc}
Every finite-rank order continuous lattice homomorphism attains its norm. In particular, every finite-rank lattice homomorphism from an order continuous Banach lattice attains its norm.
\end{cor}
\begin{proof}
    Let $T:X\to Y$ be a finite-rank order continuous lattice homomorphism. Replacing $Y$ with $TX$ does not affect order continuity of $T$, and so we may assume that $Y=TX$, which is lattice isomorphic to $\R^{n}$, for some $n\in\N$ and with the appropriate lattice norm. Let $e_1^*,\ldots,e_n^*$ be the coordinate functionals on $\R^{n}$ (which we identify with $TX$). Let $x^{*}_{k}:=e_{k}^*\circ T$, $k=1,\ldots,n$, which is an order continuous lattice homomorphism on $X$. It follows that $Tx=\left(x^{*}_{1}(x),...,x^{*}_{n}(x)\right)$, for every $x\in X$. Due to surjectivity of $T$, we have $x^{*}_{k}\ne 0$, and so by Corollary \ref{cor:coordinate} there is an atom $x_k\in X$ such that $x^{*}_{k}$ is the coordinate functional of $x_k$, for $k=1,\ldots,n$. 
    Let $P_{k}$ be the band projection onto the span of $x_k$, so that
    $P_{k}x=x^{*}_{k}(x)x_k$, for every $x\in X$.    
     Let $P$ be the band projection onto $Z:=\mathrm{span}\left\{x_1,\ldots,x_n\right\}$. We have that $P_kP=P_k$, hence $x^{*}_{k}(x)x_k=P_{k}x=P_kPx$, and so $x^{*}_{k}(x)=x^{*}_{k}(Px)$, for every $x\in X$ and $k=1,\ldots, n$. It follows that $T=TP$, $TB_{X}=TB_{Z}$ and since $Z$ is finitely dimensional, $T$ attains its norm on $B_Z$.
\end{proof}

Let us consider now the case when $X$ is close to being order continuous, in the sense that $X^a$ is of finite codimension in $X$. By the \textit{order continuous part} $X^a$ of a Banach lattice $X$ we mean the largest (closed) ideal of $X$ on which the norm is order continuous \cite[Proposition 2.4.10]{Meyer}. Note that even when $X^a$ has codimension $1$ there can be a lattice homomorphism that does not attain its norm. This is the case for example for $X=c$ (the space of convergent sequences of reals) with a renorming as in item (2) after Theorem \ref{prop:strictly-positive-functional} below.

\begin{prop}\label{prop:fincodim}
    Let $X$ be a Banach lattice. If $X^a$ has codimension $n$ in $X$, then $X$ has exactly $n$ distinct norm-one lattice homomorphisms which are not coordinate functionals of atoms.
\end{prop}
\begin{proof}
Since $X^a$ is a closed ideal of codimension $n$ in $X$, the quotient $X/ X^a$ is an $n$-dimensional Banach lattice, hence lattice isomorphic to $\R^n$. Let $Q:X\to X/X^a\equiv\R^{n}$ be the quotient map, and let $x_1^*,\ldots,x_n^*$ be as in the proof of Corollary \ref{cor:finrankoc}. Let
$x^*$ be a lattice homomorphism on $X$. We will show that $x^*$ is not a coordinate functional of an atom if and only if it is a multiple of some $x_j^*$. \smallskip

First, assume that $x^*$ vanishes on $X^a$. Then $x^*$ must be a linear combination of $x_1^*,\ldots,x_n^*$. Since lattice homomorphisms are either disjoint or proportional, we conclude that $x^*$ is a positive multiple of some $x_j^*$. Also, every atom in $X$ is contained in $X^a$, hence by Proposition \ref{prop:coordinate}, we deduce that $x^*$ cannot be a coordinate functional on $X$. \smallskip

Now we assume that $\left.x^*\right|_{X^a}\neq 0$, so that $\left.x^*\right|_{X^a}$ is a non-trivial lattice homomorphism on the order continuous Banach lattice $X^a$. By Corollary \ref{cor:coordinate}, there exists an atom $x_0\in X^a$ such that $x^*(x_0)>0$. As $X^a$ is an ideal in $X$, it follows that $x_0$ is an atom in $X$ as well. Now, by Proposition \ref{prop:coordinate}, we conclude that $x^*$ must be a coordinate functional on $X$.
\end{proof}

\section{Stability of norm attainment of lattice homomorphism under renormings}\label{section:Stability of norm attainment of lattice homomorphism under renormings}

In view of the Proposition \ref{prop:coordinate}, one might wonder what happens to a lattice homomorphism which is not a coordinate functional. Can we always find a lattice norm so that such a lattice homomorphism does not attain its norm? The answer will be affirmative in many occasions:

\begin{thm}\label{prop:strictly-positive-functional}
Let $X$ be a Banach lattice which has a strictly positive functional $\mu$. If we renorm $X$ with $\|\cdot\|_\mu:=\|\cdot\|+\mu\bigl(|\cdot|\bigr)$, then the only lattice homomorphisms attaining their norms are coordinate functionals of atoms. 
\end{thm}

\noindent Recall that a linear functional $\mu:X\to\mathbb{R}$ is said to be \textit{strictly positive} if $\mu(x)>0$ whenever $x>0$.  Before proving this result, let us see some consequences:
\begin{enumerate}[(1)]
    \item If we equip $C[0,1]$ or $L_\infty[0,1]$ with the norm $|||f|||=\|f\|_\infty +\int_0^1 |f(t)|dt$, then no (non-trivial) lattice homomorphism is norm-attaining. \smallskip
    \item If we consider $\ell_\infty=C(\beta\mathbb{N})$ with the renorming $|||(x_n)_n|||=\|(x_n)_n\|_\infty+\sum_{n=1}^\infty \frac{1}{2^n}|x_n|$, then no evaluation $\delta_t$ for $t\in \beta\mathbb{N}\backslash\mathbb{N}$ attains its norm. \smallskip
    \item Since every separable Banach lattice has a strictly positive functional (see \cite[Proposition 1.b.15]{LinTza2}), then any separable Banach lattice can always be renormed in such a way that the only norm-attaining lattice homomorphism are the coordinate functionals. \medskip
\end{enumerate}

It was shown in \cite[Theorem 4.1]{DMCRARZ} (see also Corollary \ref{cor:coordinate}) that every lattice homomorphism on an order continuous Banach lattice attains its norm.
The authors asked whether the same result can be extended to $\sigma$-Dedekind complete Banach lattices. It should be noted that the above examples answer this question \textit{in the negative}: recall that $\ell_\infty$ and $L_\infty[0,1]$ are both Dedekind complete Banach lattices and, as this property depends only on the order of $X$ (and not on the norm), renormings (1) and (2) provide examples of Dedekind complete Banach lattices with lattice homomorphisms that do not attain their norm. \smallskip

Back to the proof of Theorem \ref{prop:strictly-positive-functional}, let us fix the notation. Let $x^*$ be a lattice homomorphism on a Banach lattice $X$. Then, $x^*$ is an atom in $X^*$, and so we denote the corresponding coordinate functional by $\lambda_{x^*}\in X^{**}$. That is, if $P_{x^*}:X^*\rightarrow X^*$ represents the band projection onto the span of $x^*$, given $\mu\in X^*$ we have $P_{x^*}(\mu)=\lambda_{x^*}(\mu)x^*$. 

\begin{lem}\label{lem:renorming & NA}
Let $x^*$ be a lattice homomorphism on a Banach lattice $X$ and let $\mu \in X^*_+$. Then:
\begin{enumerate}
    \item The norm $\|x\|_\mu:=\|x\|+\mu\bigl(|x|\bigr)$, $x\in X$, is an equivalent lattice norm on $X$, such that $\|x^*\|_\mu=\frac{1}{1+\lambda_{x^*}(\mu)}\|x^*\|$.
    \item If $x^*$ attains its $\|\cdot\|_\mu$-norm at $x_0$, then $\mu(x_0)=\lambda_{x^*}(\mu)x^*(x_0)$.
    \item If also $\mu\bot x^*$, then $\|x^*\|_\mu=\|x^*\|$, and if $x^*$ attains its $\|\cdot\|_\mu$-norm at $x_0$, then $\mu(x_0)=0$.
\end{enumerate}
\end{lem}
\begin{proof}
Without loss of generality, we may assume that $\|x^*\|=1$. Let $r:=\lambda_{x^*}(\mu)$. We first show that $\|x^*\|_\mu\geq\frac{1}{1+r}$.\medskip

Fix $\varepsilon>0$ and let $x\in B_{X_+}$ such that $x^*(x)\geq 1-\varepsilon$. As $x^*$ and $\mu-rx^*$ are disjoint, by the Riesz-Kantorovich formula there exists $y\in [0,x]$ such that $x^*(x-y)\leq \varepsilon$ and $(\mu-rx^*)(y)\leq \varepsilon$. From the first inequality, we deduce
$$
\varepsilon\geq x^*(x-y)=x^*(x)-x^*(y)\geq 1-\varepsilon -x^*(y)$$
which means that $x^*(y)\geq 1-2\varepsilon$. Now, from the second one (keeping in mind that $\|y\|\leq 1$ because $x\in B_{X_+}$), we get
$$
\varepsilon\geq \mu(y)-rx^*(y)\geq \mu(y)-r
$$
or equivalently, $\mu(y)\leq \varepsilon+r.$
Hence, we have
$$
\|x^*\|_\mu \geq \frac{x^*(y)}{\|y\|_\mu}=\frac{x^*(y)}{\|y\|+\mu(y)}\geq \frac{1-2\varepsilon}{1+\varepsilon+r},
$$
and since $\varepsilon>0$ is arbitrary we conclude that $\|x^*\|_\mu\geq\frac{1}{1+r}$.\medskip

Now suppose that $|x^*(x)|=1$ for some $x\in X$. Note that $x^*(|x|)=|x^*(x)|=1$, and $\mu\geq rx^*$, hence
$$
\|x\|_\mu=\|x\|+\bigl(\mu-rx^*\bigr)(|x|)+rx^*(|x|)\geq |x^*(x)|+rx^*(|x|)=1+r.
$$

This shows that $\|x^*\|_\mu\leq\frac{1}{1+r}$. Moreover, if $x_0\in X_{+}$ is such that $|x^*(x_0)|=1$ and $\|x_0\|_\mu=1+r$, then $\bigl(\mu-rx^*\bigr)(x_0)=0$, hence $\mu(x_0)=r=\lambda_{x^*}(\mu)$. This justifies the second claim in (ii). (iii) is a special case of (i) and (ii), since $\mu\perp x^*$ is equivalent to $r=\lambda_{x^*}(\mu)=0$.
\end{proof}

\begin{proof}[Proof of Theorem \ref{prop:strictly-positive-functional}]
Suppose $\mu$ is a strictly positive functional and consider the renorming
$$
\|x\|_\mu:=\|x\|+\mu\bigl(|x|\bigr),\qquad x\in X.
$$
Let $x^*$ be a lattice homomorphism on $X$ which is not a coordinate functional of an atom and assume that $x_0\in X_{+}$ is such that $\|x_0\|_\mu =1$ and $x^*(x_0)=\|x^*\|_{\mu}$. Then, according to Lemma \ref{lem:renorming & NA} we have that $\mu(x_0)=rx^*(x_0)$, where $r\ge 0$ is such that $\mu-rx^*\geq 0$ and $\mu-rx^*\bot x^*$. It follows from Proposition \ref{prop:coordinate} that $x_0$ is not an atom. Hence, there are disjoint non-zero $u,v\in [0,x_0]$. We have that $x^*(u)\wedge x^*(v)=x^*(u\wedge v)=0$, and so we may assume that $x^*(u)=0$. From strict positivity we have 
\[\mu(x_0)>\mu(x_0-u)\ge rx^*(x_0-u)=rx^*(x_0)-rx^*(u)=rx^*(x_0),\]
which is a contradiction.
\end{proof}

\begin{cor}\label{cor:converse-coordinate functionals}
  Let $X$ be a Banach lattice and let $x^*\in \text{Hom}(X,\mathbb{R})$, $x^*\neq 0$. Suppose that there exists $x_0\in X$ such that for every lattice renorming $|||\cdot|||$ of $X$, $x^*$ attains its norm at $\frac{x_0}{|||x_0|||}$. Then $x_0$ is an atom in $X$.
\end{cor}
\begin{proof}
Suppose that $x_0$ is not an atom in $X$. Arguing as in the proof of Theorem \ref{prop:strictly-positive-functional} we can find a non-zero $y_0\in[0,x_0]$ such that $x^*(y_0)=0$. Let $\mu\in X_+^*$ be such that $\mu(y_0)>0$ and define $\|\cdot\|_{\mu}$ as in Lemma \ref{lem:renorming & NA}. Let $r\geq 0$ be such that $\mu-rx^*\geq 0$ and $\mu-rx^*\bot x^*$. By part (ii) of Lemma \ref{lem:renorming & NA} we have $\mu(x_0)=rx^*(x_0)$, hence 
\[\mu\left(y_0\right)= \left[\mu-rx^*\right](y_0)+rx^*(y_0)\leq  \left[\mu-rx^*\right](x_0)=0,\]
which is a contradiction.
\end{proof}

In general, there can be lattice homomorphisms that are not coordinate functionals but attain their norm for any lattice renorming, as shown in the next example.

\begin{example}\label{ex:omega1}
  Let $\omega_1$ denote the first uncountable ordinal, and let $[0,\omega_1]$ denote the corresponding interval of ordinals which is compact for the order topology. It is clear that the evaluation functional $\delta_{\omega_1}\in C[0,\omega_1]^*$ is a lattice homomorphism which is not a coordinate functional of an atom. Let us show that if $|||\cdot|||$ is a lattice norm on $C[0,\omega_1]$, then $\delta_{\omega_1}$ attains its norm. Assume that $|||\delta_{\omega_1}|||=1$. Then, for every $n\in\mathbb{N}$ there is $f_n\in C[0,\omega_1]^+$ such that $f_n(\omega_1)=1$ and $1\leq |||f_n|||\leq 1+\frac{1}{n}$. Since continuous functions on $[0,\omega_1]$ are eventually constant, for every $n$ there is $\alpha_n<\omega_1$ such that $f_n$ is identically $1$ on $[\alpha_n,\omega_1]$. Now, define $\alpha:=\sup_{n\in\mathbb{N}}\alpha_n$, which is again a countable ordinal. Since $(\alpha,\omega_1]$ is a clopen set, $f=\chi_{(\alpha,\omega_1]}\in C[0,\omega_1]$. Note that $0\leq f\leq f_n$ for every $n\in\mathbb{N}$, therefore $|||f|||\leq 1$ and $f(\omega_1)=1$.
\end{example}

We have already pointed out just after Theorem \ref{prop:strictly-positive-functional} that being Dedekind complete does not ensure that every lattice homomorphism attains its norm. But what happens if we assume that the Banach lattice is $\sigma$-order continuous? In the following example we show that this condition does not guarantee that $\text{Hom}(X,\mathbb{R})\subset\text{NA}(X,\mathbb{R})$ either.

\begin{example}
 Let $\Gamma$ be an uncountable set endowed with the discrete topology. If $K=\alpha\Gamma=\Gamma\cup\{\infty\}$ is the one-point compactification of $\Gamma$, then $C(K)$ is $\sigma$-order continuous but not order continuous. Indeed, suppose that $\{f_n\}_{n=1}^\infty$ is a decreasing sequence in $C(K)$ such that $\bigwedge_{n=1}^\infty f_n=0$. We claim that $f_n(\infty)\to 0$. If not, there is $\varepsilon_0>0$ such that $f_n(\infty)\geq \varepsilon_0$. For every natural $n$, $S_n=\text{supp}\bigl(f_n-f_n(\infty)\bigr)$ is countable, so $S=\cup_{n=1}^\infty S_n$ is also countable. Now, fix any $\gamma_0\in\Gamma\backslash S$ (this set is non-empty as $\Gamma$ is uncountable) and note that $f_n(\gamma_0)=f_n(\infty)\geq \varepsilon_0$ for all $n\in\mathbb{N}$. Thus, we arrive at the following contradiction
$$
0=\bigwedge_{n=1}^\infty f_n\geq \varepsilon_0\,\mathbf{1}_{\{\gamma_0\}}>0,
$$
and this implies that $f_n(\infty)\to 0$. With the latter in mind, it is easy to conclude that $\|f_n\|_\infty\to 0$. Let $\varepsilon>0$ and $N$ be such that $f_N(\infty)<\varepsilon$. Then, $\{\gamma\in\Gamma\::\: |f_N(\gamma)|\geq \varepsilon\}=\{\gamma_1,\ldots,\gamma_k\}$, and since $f_n(\gamma_i)\to 0$ for every $i=1,\ldots,k$, we find $N'\geq N$ such that $f_{N'}(\gamma_i)<\varepsilon$ for all $i$. Then $\|f_n\|_\infty\leq \varepsilon$ whenever $n\geq N'$.

Now, take an infinite sequence of points $(t_n)_{n=1}^\infty$ in $\Gamma\subset K$ and define
$$
|||f|||=\|f\|_\infty+\sum_{n=1}^\infty\frac{1}{2^n}|f(t_n)|,\qquad f\in C(K),
$$
which is a lattice renorming of $C(K)$. Therefore, $(C(K),|||\cdot|||)$ is a $\sigma$-order continuous Banach lattice which has a lattice homomorphism, specifically $\delta_\infty$, which does not attain its norm.
\end{example}

 \section{Norm attainment of lattice homomorphisms on AM-spaces}\label{sec:AMspaces}
 
Although the set of positive functionals over a Banach lattice is always large since any element of the dual is a difference of positive functionals, the set of lattice homomorphisms might be trivial for some Banach lattices. For instance, $L_p[0,1]$ does not have non-zero lattice homomorphisms whenever $1\leq p < \infty$ (see \cite[Lemma 2.31 (1)]{AA-book}).
A natural class of Banach lattices with a \textit{large set} of lattice homomorphisms is the class of AM-spaces (to get a more accurate idea, see Proposition \ref{prop:norming}).
Recall that an \textit{AM-space} is a Banach lattice $X$ such that the norm satisfies
		$$
		\|x \vee y\| = \max\{\|x\|, \|y\|\} \quad \text{for all } x, y \in X^+.
		$$
        
AM-spaces and $C(K)$-spaces are closely related. Specifically, a well-known theorem due to Kakutani asserts that a Banach lattice $X$ is an AM-space if and only if it embeds as a closed sublattice of some $C(K)$-space \cite{Kakutani-AM spaces}. It is also well known that on $C(K)$ spaces, lattice homomorphisms correspond to (positive multiples) of point evaluations. Therefore, on $C(K)$, every lattice homomorphism attains its norm; in fact, all of them attain their norm at the strong unit $\mathbf{1}_K$. However, an AM-space in general does not necessarily have a strong unit. \smallskip

Similarly, every positive functional on a $C(K)$-space attains its norm at the constant function $\mathbf{1}_K$. It turns out that in the separable setting, this property characterizes $C(K)$-spaces among AM-spaces \cite[Proposition 19.26]{OT}: given a separable AM-space $X$, every positive functional on $X$ attains its norm if and only if $X$ is lattice isometric to a $C(K)$-space. The following simple example shows that this cannot be generalized to the non-separable case.

\begin{example}
Let $\Gamma$ be an uncountable set and consider the closed sublattice of $\ell_\infty(\Gamma)$
$$
X=\bigl\{f\in\ell_\infty(\Gamma)\::\: \text{supp}(f)\mbox{ is countable}\bigr\}.
$$
Note that $X$ is a non-separable AM-space, but it does not admit a strong unit, so it cannot be lattice isometric to a $C(K)$-space. Let $x^*\in X^*$ be a norm-one positive functional. Then, for every $n\in\mathbb{N}$, there is $f_n\in B_{X_+}$ such that $x^*(f_n)\geq 1-\frac{1}{n}$. By definition of $X$, for every natural $n$, the set $S_n:=\text{supp}(f_n)$ is countable, so the union $S=\cup_{n=1}^\infty S_n$ is also countable. Therefore, $\mathbf{1}_S\in X$ and since $\mathbf{1}_S\geq f_n$ for every $n\in\mathbb{N}$, this implies that $x^*(\mathbf{1}_S)=1$.
\end{example}

Lattice homomorphisms, as opposed to positive functionals, have a similar behavior on $C(K)$-spaces and general AM-spaces in terms of norm attainment. We will provide two proofs of this fact: the one in Section \ref{section:Free Banach lattices over Banach spaces} uses some advanced tools coming from free Banach lattices, whereas the proof presented in this section is elementary, only requiring the following characterization of norm-attaining lattice homomorphisms. This characterization is interesting on its own and, in fact, it will also be decisive to prove that every lattice homomorphism on a free Banach lattice over a lattice attains its norm (Proposition \ref{prop:normattainingFBLlattice}).

\begin{prop}\label{prop:characterizationLatNA}
Let $X$ be a Banach lattice and $x^*\in X^*$ a lattice homomorphism with $\|x^*\|=1$. Then $x^*\in\text{NA}(X,\mathbb{R})$ if and only if  there exists an increasing sequence $(x_n)_{n=1}^\infty$ in $B_{X_+}$ such that $x^*(x_n)\to 1$ as $n\to\infty$.
\end{prop}
 \begin{proof}
Necessity is trivial, by considering the constant sequence defined by $x_n:=x\in B_{X_+}$ where $x^*(x)=1$. Let us prove sufficiency. Let $(x_n)_{n=1}^\infty \subset B_{X^+}$ be an increasing sequence such that $0<x^*(x_n)\to 1$. It follows that $x^*(x_n)$ is increasing. Define
 $$
 y_n:=\bigwedge_{k=1}^n \frac{x_k}{x^*(x_k)}, \quad \text{ for } n\in\mathbb{N}.
 $$
The above sequence $(y_n)_{n=1}^\infty$ is clearly decreasing and we claim that it is a Cauchy sequence in $X$. Given $m,n\in\mathbb{N}$, $m>n$, we have
\begin{eqnarray*}
0 &\leq& y_n-y_m=y_n - y_n\land \left(\bigwedge_{k=n+1}^m \frac{x_k}{x^*(x_k)} \right)=\left( \bigwedge_{k=1}^n \frac{x_k}{x^*(x_k)} - \bigwedge_{k=n+1}^m \frac{x_k}{x^*(x_k)}\right)^+ \\
&\leq &  \left( \frac{x_n}{x^*(x_n)}- \bigwedge_{k=n+1}^m x_{k}\right)^+ \leq \left( \frac{x_{n+1}}{x^*(x_n)}- x_{n+1}\right)^+ = \left(\frac{1}{x^*(x_n)}-1\right) x_{n+1}.
\end{eqnarray*}
By the monotonicity of the norm of $X$ we obtain $\|y_n-y_m\|\leq \frac{1}{x^*(x_n)}-1$, and this shows that $(y_n)_{n=1}^\infty$ is Cauchy. Let us denote by $y$ the limit of this sequence in $X$ and note that for every $n\in\N$,
$$
x^*(y_n)=\bigwedge_{k=1}^n \frac{x^*(x_k)}{x^*(x_k)}=1, \quad \text{ and } \quad \|y_n\|\leq \frac{\|x_n\|}{x^*(x_n)}\leq \frac{1}{x^*(x_n)},$$
so $x^*(y)=1=\|y\|$.
 \end{proof}

As a consequence, we get:

\begin{thm}
\label{THMAM}
    Every lattice homomorphism on an AM-space attains its norm.
\end{thm}
\begin{proof}
Let $x^*$ be a norm-one lattice homomorphism on an AM-space $X$. Then, for every $n\in\mathbb{N}$, there is $y_n\in B_{X^+}$ such that $x^*(y_n)\geq 1-\frac{1}{n}$. Now, define for every natural $n\geq 1$, $x_n:=\bigvee_{k=1}^n y_n$. Clearly, $(x_n)_{n=1}^\infty$ is an increasing sequence in $X$. Moreover, since $X$ is an AM-space, $x_n\in B_{X^+}$, and by the positivity of $x^*$, we also have $x^*(x_n)\geq 1-\frac{1}{n}$ for $n \in\mathbb N$. Consequently, by Proposition \ref{prop:characterizationLatNA}, $x^*$ attains its norm.
\end{proof}

As in the order continuous case, the previous theorem can be generalized for lattice homomorphisms of finite rank.

\begin{cor}
 Let $X$ be an AM-space and $Y$ an arbitrary Banach lattice. Then, every finite-rank lattice homomorphism $T:X\to Y$ attains its norm.   
\end{cor}

\begin{proof}
We may assume that $Y=T(X)$. Hence, $Y$ is a finite-dimensional Banach lattice, so its order is determined by a $1$-unconditional basis $\{y_k,y_k^*\}_{k=1}^n$. Thus, for every $x\in X$, we can write
    $$
    Tx=\sum_{k=1}^n y_k^*(Tx)y_k=\sum_{k=1}^n x_k^*(x)y_k,
    $$
    where $x_k^*:=y_k^*\circ T$ for $1\leq k\leq n$. Note that for every $x\in B_{X_+}$,
    $$
    Tx=\sum_{k=1}^n x_k^*(x)y_k\leq \sum_{k=1}^n \|x_k^*\| y_k=:y_0,
    $$
    so $\|T\|\leq \|y_0\|$. On the other hand, as each $x_k^*$ is a lattice homomorphism on the AM-space $X$, by Theorem \ref{THMAM}, we can find $x_k\in B_{X_+}$ such that $x_k^*(x_k)=\|x_k^*\|$ for all $k$. Now, we define $x_0:=\lor_{k=1}^n x_k$, which belongs to $B_{X_+}$ as $X$ is an AM-space. Therefore, $Tx_0=y_0$ and, in particular, this shows that $T$ attains its norm at $x_0$.    
\end{proof}

The following is an example of a Banach lattice $X$ where every lattice homomorphism in $X^*$ attains its norm, but finite-rank lattice homomorphisms need not attain their norm.

\begin{example}\label{finite-rankLH}
Let us consider the following norm in $c\oplus c$:
$$
\|(x_n,y_n)\|=\max\left\{\|(x_n)\|_\infty+\sum_{n=1}^\infty \frac{1}{2^n}|y_n|, \:\: \|(y_n)\|_\infty\right\}.
$$ \smallskip
It is clear that $\|(x_n,y_n)\|_\infty \leq \|(x_n,y_n)\|\leq 2\|(x_n,y_n)\|_\infty$. Note that every lattice homomorphism on this Banach lattice attains its norm: for example, $e_n^*\oplus 0$ attains its norm at $(e_n,\mathbf{0})$ and $e_\infty^*\oplus 0$ at $(\mathbf{1},\mathbf{0})$ (where $\mathbf{0}$ and $\mathbf{1}$ represent the sequences which are constantly $0$ and $1$, respectively).
The operator $T:(c\oplus c,\|\cdot\|)\to\ell_1^2$ defined by \smallskip
$$
T(x_n,y_n):=\bigl(\lim_{n}x_n,\lim_n y_n\bigr)=(e_\infty^*\oplus 0)(x_n,y_n)\,e_1+(0\oplus e_\infty^*)(x_n,y_n)\,e_2
$$
is clearly a lattice homomorphism of norm at most $2$. To see that $\|T\|=2$ denote $z^n=\mathbf{1}-\sum_{k=1}^n e_k \in c$, and consider the sequence
$(\mathbf{1},z^n)\in c\oplus c$. We have that $\|T(\mathbf{1},z^n)\|_1=2$ for every $n$, whereas $\|(\mathbf{1},z^n)\|\to 1$ as $n\to\infty$. On the other hand, suppose that there exists a positive element $(x,y)\in c\oplus c$ with $\|(x,y)\|=1$ such that $\|T(x,y)\|_1=2$. This implies that $\lim_n x_n=\lim_n y_n=1$. In particular, $(y_n)$ cannot be the constant sequence $\mathbf{0}$, so $\|(x_n)\|_\infty+\sum_{n=1}^\infty \frac{1}{2^n}|y_n|>1$ and  $\|(x,y)\|>1$.
\end{example}

Another easy consequence of Theorem \ref{THMAM} is the following version of Urysohn's lemma for AM-spaces (see Proposition \ref{prop:norming} for a more topological formulation of the statement).

\begin{cor}\label{cor:Urysohn like}
    Let $X$ be an AM-space and let $x^*_1,\ldots,x^*_n$ be distinct lattice homomorphisms on $X$ of norm $1$. For every $a_1,\ldots,a_n\in [0,1]$, there exists $x\in B_X$ such that $x^*_k(x)=a_k$ for every $k=1,\ldots,n$.
\end{cor}
\begin{proof}
    It suffices to show that there is an element $x\in B_X$ such that $x^*_1(x)=1$ and $x^*_k(x)=0$ for $k=2,\ldots, n$. By Theorem \ref{THMAM}, there exists $y\in B_{X^+}$ such that $x^*_1(y)=1$. 
    By Kakutani's representation theorem, there is a lattice homomorphism $T:C(K)\to X$, for some compact Hausdorff $K$, such that $T\mathbf{1}_K=y$. Therefore, $T^*x^*_1,\ldots,T^*x^*_n$ are lattice homomorphisms on $C(K)$. Since $x^*_1,\ldots,x^*_n$ are linearly independent, they are disjoint; since lattice operations on $X^*$ commute with taking the restriction to an ideal, it follows that $T^*x^*_1,\ldots,T^*x^*_n$ are disjoint.
    
    There exist $r_1,\ldots,r_n\ge 0$ and  $t_1,\ldots,t_n\in K$ such that $T^*x^*_k=r_k\delta_{t_k}$. Note that $1=x^*_1(y)=[T^*x^*_1](\mathbf{1}_K)=r_1 \mathbf{1}_K(t_1)=r_1$. For those $k\ne 1$ for which $r_k\ne 0$, since $r_k\delta_{t_k}\bot\delta_{t_1}$ we have $t_k\ne t_1$. By Urysohn's lemma there is a continuous function $f:K\to [0,1]$ such that $f(t_1)=1$ and $f(t_k)=0$ for $k=2,\ldots, n$ with $r_k>0$. Then, putting $x=Tf$ yields $x^*_1(x)=f(t_1)=1$ and $x^*_k(x)=r_k f(t_k)=0$, for all $k=2,\ldots, n$.
\end{proof}

\section{Lattice homomorphisms on AM-spaces and renormings}\label{section:renormings AMspaces}

Next, we continue with the study initiated in Section \ref{section:Stability of norm attainment of lattice homomorphism under renormings} of norm attainment of lattice homomorphisms after renormings, this time in the class of AM-spaces. The next proposition illustrates that it is a highly non-isomorphic property, at least for $C(K)$-spaces.

    \begin{prop}
\label{PropNotNA}
    Let $K$ be an infinite compact space. Then $C(K)$ has an equivalent lattice norm with a non-norm-attaining lattice homomorphism.
\end{prop}
\begin{proof}
Take a sequence $(t_n)_n$ in $K$ with an accumulation point $t \in \overline{\{t_n:n \in \N \}}\setminus\{t_n:n \in \N \}$.
Consider the equivalent lattice norm
$$
    |||f|||=\|f\|_\infty+ \sum_{n=1}^\infty \frac{1}{2^n}|f(t_n)|.
    $$
Then the evaluation functional $\delta_t \in C(K)^*$ satisfies $|||\delta_t|||=1$.
However, by definition of $t$ and the sequence $(t_n)_n$ we have that if $f(t)=1$ then $|||f||| >1$, so $\delta_t$ is not norm-attaining.
\end{proof}

Note that Example \ref{ex:omega1} illustrates that the previous argument does not imply that for every accumulation point $t$ of the compact $K$ we can find a renorming such that the evaluation functional $\delta_t$ does not attain its norm. 

\begin{prop}\label{prop:evaluation}
Let $K$ be compact and Hausdorff, and let $X$ be a (not necessarily closed) sublattice of $C(K)$. Then:
\begin{enumerate}
\item $\overline{X}=\Bigl(\bigl\{\delta_s-\gamma\delta_t \mid s,t\in K,\           \gamma\ge 0\bigr\} \cap X^\perp\Bigr)_\perp$. In particular, $X$ is a dense sublattice of a closed sublattice $Y\subset C(K)$ if and only if whenever for $t,s\in K$ the functionals $\left.\delta_s\right|_{Y}$ and $\left.\delta_t\right|_{Y}$ are linearly independent, $\left.\delta_s\right|_{X}$ and $\left.\delta_t\right|_{X}$ are also linearly independent.
\item If $X$ is closed, then for every $x^*\in \text{Hom}(X,\mathbb{R})$ there are $t\in K$ and $r\ge 0$ such that $x^*=\left.r\delta_{t}\right|_{X}$.
\end{enumerate}
\end{prop}
\begin{proof}
(i) see in e.g. \cite[Theorem 2.1]{BT}.

(ii): We may assume that $x^*\ne 0$, and so $N:=\ker x^*$ is a proper closed ideal in $X$. Let $H$ be the (not necessarily closed) ideal generated by $N$ in $C(K)$, that is,
$$
H:=\bigl\{h\in C(K)\::\: \exists\, x\in N \text{ such that } |h|\leq x\bigr\}
$$
We claim that $\overline{H}\cap X=N$. Indeed, if $x\in X_+$ and $(h_n)_{n=1}^\infty\subset H$ are such that $h_n\to x$, then $x\wedge h_n^{+}\to x$. For every $n\in\N$ there is $x_n\in N$ such that $x_n\geq h_n$, hence we have $|x-x\wedge x_n^{+}|\le |x-x\wedge h_n^{+}|\to 0$, and so $x\wedge x_n^{+}\to x$, as $n\to\infty$. As $x\wedge x_n^{+}\in N$ for every $n\in\N$, we conclude that $x\in N$.

As $\overline{H}$ is a proper closed ideal in $C(K)$, there is a closed $A\subset K$ such that $h\in \overline{H}$ if and only if $h$ vanishes on $A$ (see e.g. \cite[Proposition 2.1.9]{Meyer}). Since $X\not\subset \overline{H}$, there is $a\in A$ and $x\in X$ such that $x(a)\ne 0$. It follows that $\ker\delta_{a}\cap X$ is a closed subspace in $X$ of codimension $1$. Since we also have $N\subset H\subset \ker \delta_{a}$, we conclude that $\ker x^*=N=\ker\delta_{a}\cap X$, and so $x^*$ is a multiple of $\left.\delta_{a}\right|_{X}$.
\end{proof} 

We present next an alternative proof of Proposition \ref{prop:evaluation} (ii). It is a different approach and provides some extra information. 

\begin{prop}
    Let $K$ be a compact Hausdorff space and let $X$ be a (closed) sublattice of $C(K)$. If $x^*$ is a norm-one lattice homomorphism on $X$, then there is $t\in K$ such that $x^*=\left.\delta_t\right|_X$. Thus, for every $x^*\in\text{Hom}(X,\mathbb{R})$ there is $t\in K$ such that $x^*=\|x^*\|\left.\delta_t\right|_X$.
\end{prop}
\begin{proof}
First, we consider the set $\mathcal{E}_{x^*}$ of norm preserving positive extensions of $x^*$ to $\mathcal{C}(K)$, that is,
$$\mathcal{E}_{x^*}=\{y^*\in \mathcal{C}(K)^*\::\: \|y^*\|\leq1, \quad y^*\geq 0, \quad \left.y^*\right|_X=x^*\},$$
which is non-empty by \cite[Corollary 1.3]{Lotz}. Moreover, since $\mathcal{E}_{x^*}$ is $w^*$-closed and convex, we deduce from Krein-Milman theorem that the set of its extreme points is non-empty. 

Fix an extreme point $y_0^*$ of $\mathcal{E}_{x^*}$. We will prove that $y_0^*$ is also an extreme point of the set $\{y^*\in B_{\mathcal{C}(K)^*}\::\: y^*\geq 0\}$. Let $y_1^*,\,y_2^*\in B_{C(K)_+^*}$ be such that $y_0^*=\frac{1}{2}y_1^*+\frac{1}{2}y_2^*$. If we restrict these functionals to $X$, we obtain
$$
x^*=\left.y_0^*\right|_X=\frac{1}{2}\left.y_1^*\right|_X+\frac{1}{2}\left.y_2^*\right|_X.
$$
Since $x^*$ is an atom in $X^*$, we deduce from the above identity that $\left.y_1^*\right|_X=\left.y_2^*\right|_X=x^*$. This shows that $y_1^*,\,y_2^*\in \mathcal{E}_{x^*}$ and given that $y_0^*$ is an extreme point of $\mathcal{E}_{x^*}$, we have $y_0^*=y_1^*=y_2^*$. Therefore, $y_0^*$ is a non-trivial extreme point of $\{y^*\in B_{\mathcal{C}(K)^*}\::\: y^*\geq 0\}$, so there exists $t\in K$ such that $y_0^*=\delta_t$.
\end{proof}

Let $E$ be a Banach space and $A\subset B_{E^*}$ be a \emph{positively balanced} set, that is, a set such that $rx^*\in A$ for every $x^*\in A$ and $r\in[0,1]$. Let $C_{ph}(A)$ be the space of all weak$^*$ continuous functions $f$ on $A$ which are \emph{positively homogeneous}, i.e. $f(rx^*)=rf(x^*)$, for all $r\in[0,1]$ and $x^*\in A$. If the subset $A$ is weak$^*$ compact, $C_{ph}(A)$ endowed with the supremum norm is an AM-space, but in general it does not have a strong unit. We say that a subset $A\subset E^*$ is $\lambda$-norming, for some $\lambda\geq 1$, if $\sup\{|x^*(x)|\::\: x^*\in A\cap B_{E^*}\}\geq \frac{1}{\lambda} \|x\|$ for every $x\in E$. For a Banach lattice $X$ let $K_X:=\text{Hom}(X,\mathbb{R})\cap B_{X^*}$, endowed with the weak$^*$ topology, making it a compact space. The main role that renormings of AM-spaces play in our study is due to the following proposition. Moreover, it provides an explicit representation for an AM-space $X$ as $C_{ph}(K_X)$ that will be useful for our purposes. The result is mentioned in \cite[Corollaire 1.31]{GdR}, but despite its similarities with Kakutani's Theorem \cite{Kakutani-AM spaces}, it does not seem to be very well known.

\begin{prop}\label{prop:norming}
A Banach lattice $X$ is $\lambda$-lattice isomorphic to an AM-space if, and only if, the set $K_X$ is $\lambda$-norming. In this case it is $\lambda$-lattice isomorphic to $C_{ph}(K_{X})$.
\end{prop}
\begin{proof}
First, note that if $X$ is a closed sublattice of $C(K)$, for some $K$, then 
$$
\|x\|=\sup_{t\in K}|x(t)|=\sup_{t\in K} \bigl|\left.\delta_t\right|_X(x)\bigr|\leq \sup_{x^*\in K_X}|x^*(x)|\leq \|x\|,
$$
and so $K_X=\text{Hom}(X,\mathbb{R})\cap B_{X^*}$ is $1$-norming. \smallskip

Suppose that $X$ is $\lambda$-lattice isomorphic to an AM-space $Y$, for some $\lambda\geq 1$. Let $T\colon X\to Y$ be a lattice isomorphism such that $\|T\|\|T^{-1}\|\leq \lambda$. As $K_{Y}=\text{Hom}(Y,\mathbb{R})\cap B_{Y^*}$ is $1$-norming, for every $x\in X$ we have that
\begin{eqnarray*}
\sup \{|x^*(x)|\::\: x^*\in K_X\} &= & \sup \big\{|(T^{-1})^*x^*(Tx)|\::\: x^*\in K_X\bigr\} \\
& \geq & \frac{1}{\|T\|} \sup \{|y^*(Tx)|\::\: y^*\in K_Y\} \\
&\geq & \frac{1}{\|T\|} \|Tx\| \geq \frac{1}{\|T\|\|T^{-1}\|} \|x\|\geq \frac{1}{\lambda}\|x\|. 
\end{eqnarray*}

\smallskip Conversely, let us suppose that $\text{Hom}(X,\mathbb{R})$ is $\lambda$-norming. Define $J:X\to C_{ph}(K_{X})$ by $[Jx](x^*):=x^*(x)$. It is easy to see that $Jx$ is in fact an element of $C_{ph}(K_{X})$. Clearly, $J$ is a lattice homomorphism, and it follows from our assumption that $\frac{1}{\lambda}\|x\|\le \|Jx\|\le \|x\|$. Hence, $JX$ is a closed sublattice of $C_{ph}(K_{X})$. We finish proving that $JX$ is dense in $C_{ph}(K_{X})$. According to part (i) of Proposition \ref{prop:evaluation} it is enough to show that if $x^*$ and $y^*$ in $K_X$ are not proportional, then there is $x\in X$ such that $(Jx)(x^*)=0\ne (Jx)(y^*)$, which is trivial. We conclude that $X$ is $\lambda$-lattice isomorphic to $C_{ph}(K_{X})$.
\end{proof}

\begin{rem}
Note that if a Banach lattice $X$ is merely isomorphic (as a Banach space) to an AM-space, then it is already \textit{lattice} isomorphic to some AM-space (this is a consequence of \cite[Corollary 2.2]{HMST}).  In general, it is possible for two Banach lattices to be isometric as Banach spaces, while their collections of lattice homomorphisms \textit{differ completely}: $\ell_2$ and $L_2[0,1]$ are linearly isometric and we know that $\text{Hom}(\ell_2,\mathbb{R})=\{\lambda e_n^*\::\: \lambda\geq 0,\:n\in\mathbb{N}\}$, whereas $\text{Hom}(L_2[0,1],\mathbb{R})=\{0\}$. Also, the proposition fails if $\text{Hom}(X,\mathbb{R}$) is merely total (i.e., $x^*(x)=0$ for every $x^*\in \text{Hom}(X,\mathbb{R})$ if and only if $x=0$), and not norming: consider $X=\ell_p$, for $p\in[1,+\infty)$.
\end{rem}

We can now extend Proposition \ref{PropNotNA} to general AM-spaces, and in the process recover a classical result (\cite[Lemma 1.b.10]{LinTza2}).

\begin{thm}\label{thm:several}
Given an AM-space $X$, the following conditions are equivalent:
\begin{enumerate}
\item $X$ is lattice isometric to $c_0(\Gamma)$ for some cardinal $\Gamma$;
\item $X$ is order continuous;
\item Every lattice homomorphism on $X$ is a multiple of the coordinate functional of some atom;
\item Every lattice homomorphism on $X$ attains its norm for every lattice renorming of $X$; %%%
\item For every $x\in X$ and $\varepsilon>0$ there are only finitely many lattice homomorphisms $x^*$ of norm $1$ with $|x^*(x)|\geq \varepsilon$.
\end{enumerate}
\end{thm}
\begin{proof}
(i)$\Rightarrow$(ii) is clear, (ii)$\Rightarrow$(iii) follows from Corollary \ref{cor:coordinate}, (iii)$\Rightarrow$(iv) follows from Proposition \ref{prop:coordinate}.\medskip

(iv)$\Rightarrow$(v): Assume that there is $\varepsilon_0>0$, $x_0\in X_+$ and an infinite sequence $(y_n^*)_{n=1}^\infty$ of distinct homomorphisms in $S_{X^*}$ such that $y^*_n(x_0)\geq \varepsilon_0$, for every $n\in\N$. As $\text{Hom}(X,\mathbb{R})\cap B_{X^*}$ is weak$^*$-compact, we can find a lattice homomorphism $z^*\in B_{X^*}$ in $\overline{(y_n^*)_{n=1}^\infty}^{w^*}$. Note that $z^*\ne 0$ since it must fulfill $z^*(x_0)\geq \varepsilon_0>0$. We may assume that $y_n^*$ is non-proportional, and hence disjoint with $z^*$, for every $n\in\N$. Let $\mu:=\sum_{n=1}^\infty\frac{1}{2^n}y^*_n\in X^*$, which is disjoint with $z^*$. If $z_0\in X_{+}$ is such that $\|z_0\|_\mu=1$ and $z^*(z_0)=\|z^*\|_{\mu}$, then by Lemma \ref{lem:renorming & NA} we have $\mu(z_0)=0$. On the other hand, as $z^*(z_0)>0$ and $z^*\in \overline{(y_n^*)_{n=1}^\infty}^{w^*}$ there is $n\in\N$ so that $\mu(z_0)\geq 2^{-n} y_n^*(z_0)>0$. Contradiction.\medskip

(v)$\Rightarrow$(i) Let us enumerate the set $\text{Hom}(X,\mathbb{R})\cap S_{X^*}=\{x^*_\alpha:\alpha\in\Gamma\}$. It is immediate that $x\in X\mapsto (x_\alpha^*(x))_{\alpha\in\Gamma}\in c_0(\Gamma)$ defines a lattice embedding of $X$ into $c_0(\Gamma)$. This embedding is an isometry due to Proposition \ref{prop:norming}. In order to check the surjectivity of the mapping, by part (i) of Proposition \ref{prop:evaluation} it is enough to show that for every distinct $\beta,\gamma\in\Gamma$ there is $x\in X$ such that $x^*_\beta(x)=0\ne x^*_\gamma(x)$, which follows immediately from Corollary \ref{cor:Urysohn like}.
\end{proof}

\begin{rem}
    We could add a sixth equivalent statement to the previous theorem: $X$ is linearly isometric to $c_0(\Gamma)$. Recall that a Banach-Stone type theorem holds for AM-spaces in the sense that two AM-spaces which are linearly isometric must also be lattice isometric. As far as we know, this was first observed by Lacey \cite[Corollary 2 to Theorem 4 of Section 19, p. 188]{Lacey-book}. 
\end{rem}

One might wonder whether the above result, specifically the equivalence (iii)$\Leftrightarrow$(iv), could be generalized to arbitrary Banach lattices. This is not possible, as the following example, inspired by \cite{LW}, shows. We are going to construct a non-atomic Banach lattice $X$ which has exactly one lattice homomorphism  (up to scaling), that is norm-attaining for any lattice renorming of $X$.

\begin{example}\label{rem:counterexample-question3}
Let $\omega_1$ denote the first uncountable ordinal, and let $C([0,\omega_1],L_2[0,1])$ be the space of continuous functions on the ordinal interval $[0,\omega_1]$ with values on $L_2[0,1]$. Consider
$$
X=\bigl\{F\in C\bigl([0,\omega_1],L_2[0,1]\bigr)\::\: F(\omega_1)=a_F\mathbf{1}_{[0,1]} \text{ for some } a_F\in\mathbb{R}\bigr\}.
$$
It is clear that $x^*(F)=a_F$ is a lattice-homomorphism. For any lattice norm in $X$, it attains its norm by a similar argument as in Example \ref{ex:omega1}. Now, assume that $y^*\in X^*$ is a lattice homomorphism. For every countable limit ordinal $\alpha$ the space $C\bigl([0,\alpha],L_2[0,1]\bigr)$ embeds as a projection band in $X$, and so according to \cite[Corollary 3.5]{DMCRARZ}, $y^*$ vanishes on $C\bigl([0,\alpha],L_2[0,1]\bigr)$. Hence, $y^*$ vanishes on the kernel of $x^*$, and so the two functionals are proportional. \vspace{2pt}

Note that $X$ is atomless. Indeed, given $F\in X$, $F>0$, there exists an isolated point $\alpha\in [0,\omega_1)$ such that $F(\alpha)>0$ and keeping in mind that $L_2[0,1]$ is non-atomic, there is $g\in L_2[0,1]\cap[0,F(\alpha)]$ which is non-proportional to $F(\alpha)$; then $\mathbf{1}_{\alpha}g\in [0,F]$ is non-proportional to $F$ implying that the latter is not an atom.
\end{example}\medskip

We conclude this section with some results in the spirit of Section \ref{section:Coordinate functionals on a Banach lattice}. The following is a characterization of coordinate functionals of atoms available exclusively for AM-spaces which generalizes the (probably) well-known fact that the norm-one lattice homomorphisms on $C(K)$ of this type are precisely the evaluations $\delta_t$ at isolated points of $K$. Note that the conditions in the following result are not equivalent in general to the conditions in Proposition \ref{prop:coordinate BL}, as Example \ref{rem:counterexample-question3} shows.

\begin{prop}\label{prop:coordinate functionals are isolated}
Let $x^*$ be a lattice homomorphism on an AM-space $X$ with $\|x^*\|=1$. The following conditions are equivalent:
\begin{enumerate}
\item $x^*$ is a coordinate functional of an atom in $X$.
\item $K_X\backslash\{\lambda x^*\::\: \lambda>0\}$ is $w^*$-closed. 
\item No nonzero multiple of $x^*$ is a weak* accumulation point of $K_X\cap S_{X^*}$.
\item There is a compact Hausdorff space $K$ such that $X$ embeds as a closed sublattice of $C(K)$ in a way that $x^*=\delta_{t}|_{X}$, and $1_{\{t\}}\in X$, where $t$ is an isolated point in $K$.
\end{enumerate}    
In this case, $x^*$ is an isolated point in $K_X\cap S_{X^*}$ with respect to the weak* topology.
\end{prop}
\begin{proof}
(i)$\Rightarrow$(ii) follows from Proposition \ref{prop:coordinate BL}. (ii)$\Rightarrow$(iii) is straightforward. (iv)$\Rightarrow$(i) is obvious.

(iii)$\Rightarrow$(ii): Assume that there is $\theta\in(0,1]$ and a net $(x^*_\alpha)_\alpha \subseteq K_X\backslash\{\lambda x^*\::\: \lambda>0\}$ such that $x^*_\alpha$ weak* converges to $\theta x^*$. Without loss of generality, we may assume that the net $(\|x^*_\alpha\|)_\alpha$ converges to some scalar $\mu\in [\lambda,1]$. Therefore, the normalized net $(x^*_\alpha/\|x^*_\alpha\|)_\alpha\subseteq K_X\cap S_{X^*}\backslash\{\lambda x^*\::\: \lambda>0\}$ weak* converges to $\frac{\theta}{\mu}x^*$, which contradicts (iii).\medskip

(ii)$\Rightarrow$(iv): By Proposition \ref{prop:norming} we may assume that $X=C_{ph}(K_{X})$. Let $L:=K_X\backslash\{\lambda x^*\::\: \lambda>0\}$, which is closed by assumption, and let $K:=L\cup \left\{x^*\right\}\subset K_X$, which is therefore a compact Hausdorff space. Let $Y=\left\{f:K\to\R,~ f|_L\in C_{ph}(L)\right\}\subset C(K)$, and let $R:X\to Y$ be the restriction operator, which is easily seen to be a lattice isometry. To show its surjectivity, for $g\in Y$ consider its positively homogeneous extension $f$ defined by
    $$
    f\left(\lambda x^*\right)=\lambda, \quad \text{ for } \lambda\in(0,1).
    $$
Observe that $L$ and $M:=\bigl\{\lambda x^*\::\: \lambda\in[0,1]\bigr\}$ are $w^*$-closed subsets of $K_X$ whose union is $K_X$ and $f$ is $w^*$-continuous on each of them and is well defined in $L\cap M=\{0\}$, so $f$ is $w^*$-continuous on $K_X$. Clearly, $t:=x^*$ fulfills the requirements.\medskip

The last assertion follows immediately from (iii).
\end{proof}

The last condition is not equivalent to the rest, as the following example demonstrates.

\begin{example}
    Let $X=C[0,1]$ endowed with the norm $\|f\|:=\|f\|_{\infty}\vee 2|f(0)|$, which is isometrically isomorphic to the sublattice of $C(K)$, where $K=\{-1\}\cup [0,1]$, of functions $f$ satisfying $f(-1)=2f(0)$. We then have that $2\delta_0\in K_X\cap S_{X^*}$ is an isolated point, despite not being a coordinate functional of an atom.
\end{example}

We also have the following supplement to Proposition \ref{prop:fincodim}. Again, Example \ref{rem:counterexample-question3} shows that the implications (ii)$\Rightarrow$(i) and (iii)$\Rightarrow$(i) are not valid for general Banach lattices.

\begin{prop}\label{prop:fincodimam}
    For an AM-space $X$ the following conditions are equivalent:
    \begin{enumerate}
        \item $X^a$ has codimension $n$ in $X$.
        \item $X$ has exactly $n$ distinct lattice norm-one homomorphisms which are not coordinate functionals of atoms.
        \item The set $S'$ of weak* accumulation points of $S=K_X\cap S_{X^*}$ contains exactly $n$ linearly independent elements.
    \end{enumerate}
\end{prop}
\begin{proof}
    (i)$\Rightarrow$(ii) follows from Proposition \ref{prop:fincodim}.\medskip
    
    (ii)$\Rightarrow$(i): Assume that $x_1^*,\ldots,x_n^*$ are the distinct norm-one lattice homomorphisms which are not coordinate functionals of atoms. Since $X^{a}$ is an order continuous AM-space, by Theorem \ref{thm:several} it is lattice isometric to $c_0(\Gamma)$, so in particular, it is a closed span of its atoms (atoms in $X^a$, hence in $X$). As $x_1^*,\ldots,x_n^*$ vanish on atoms, they vanish on $X^{a}$. This shows that $X^a$ is of co-dimension at least $n$. Let $Y:=X\slash X^a$, and let $Q:X\to Y$ be the quotient map. It is enough to show that $\dim Y\le n$. Since $Y$ is an AM-space, this amounts to showing that there are at most $n$ linearly independent homomorphisms on $Y$. Let $y^{*}$ be a non-zero homomorphism on $Y$. Then, $Q^*y^*=y^{*}\circ Q$ vanishes on all atoms, and so is a positive multiple of $x^*_k$, for some $k=1,\ldots,n$. As $Q^*$ is an injection, the claim follows.\medskip    

    (ii)$\Leftrightarrow$(iii): First note that $S'\subseteq K_X$, as $K_X$ is weak* closed. According to Proposition \ref{prop:coordinate functionals are isolated}, the non-zero elements of $S'$ are precisely the homomorphisms which are not coordinate functionals of atoms (up to scaling).
\end{proof}
\begin{rem}
 Note that in the particular case of $X=C(K)$ the previous proposition shows that $C(K)^a$ has codimension $n$ in $C(K)$ if and only if $K$ has exactly $n$ non-isolated points.  
\end{rem}

\section{Free Banach lattices over Banach spaces}\label{section:Free Banach lattices over Banach spaces}

When it comes to the study of norm-attaining lattice homomorphisms, the class of free Banach lattices is certainly a relevant one. Indeed, the first explicit examples of Banach lattices with lattice homomorphisms not attaining their norm were found within this class \cite[Corollary 5.2]{DMCRARZ}. In this section, we will focus on free Banach lattices generated by Banach spaces, which are free objects in the category of Banach lattices with lattice homomorphisms with respect to the larger category of Banach spaces with bounded operators. Specifically, we will continue the line of research started in \cite{DMCRARZ}, where it was conjectured that the norm-attaining lattice homomorphisms in a free Banach lattice generated by a Banach space $E$ were in correspondence with the norm-attaining functionals in $E$.

\begin{defn}\cite{ART18}
	Let $E$ be a Banach space. The \textit{free Banach lattice generated by} $E$, denoted by $FBL[E]$, is a Banach lattice together with a linear isometric embedding $\delta_E: E \to FBL[E]$ such that for any Banach lattice $X$ and any bounded linear map $T: E \to X$, there exists a unique lattice homomorphism $\widehat{T}: FBL[E] \to X$ such that $\widehat{T} \circ \delta_E = T$ and $\|\widehat{T}\| = \|T\|$.
\end{defn}

For a given $1\leq p\leq \infty$, if in the previous definition we replace Banach lattice by $p$-convex Banach lattice (with convexity constant 1), then we get the free $p$-convex Banach lattice generated by $E$, which is denoted by $FBL^{(p)}[E]$ \cite{JLTTT}. We refer the reader to \cite{OTTT} (and the references therein) for a deep study of this class of Banach lattices. Note that $\infty$-convex Banach lattices are precisely (up to a lattice renorming) the class of AM-spaces. The free $\infty$-convex Banach lattice over a Banach space $E$, $\fbl^{(\infty)}[E]$, coincides with the Banach lattice $C_{ph}(B_{E^*})$ of positively homogeneous weak* continuous functions on $B_{E^*}$ equipped with the uniform norm (\cite[Proposition 2.2]{OTTT}), and, for $1\leq p<\infty$, $\fbl^{(p)}[E]$ can be represented as a (not necessarily closed) sublattice of $C_{ph}(B_{E^*})$. In particular, $\fbl^{(\infty)}[E]$ is an AM-space. \medskip

We now provide an analogue of Proposition \ref{PropNotNA} for free Banach lattices over Banach spaces. Recall that the lattice homomorphisms in $\fbl^{(p)}[E]$ are precisely the extensions to $\fbl^{(p)}[E]$, as lattice homomorphisms, of the functionals $x^*\in E^*$, which will be denoted by $\widehat{x^*}$ (i.e. $\widehat{x^*}\circ \delta_E=X^*$). However, these can also be seen as appropriate multiples of the evaluations on points $x^*\in B_{E^*}$, when viewing $\fbl^{(p)}[E]$ as a sublattice of $C_{ph}(B_{E^*})$.

\begin{cor}\label{cor:FBL}
Let $E$ be a Banach space of $\text{dim}(E)\geq 2$. Given a norm-attaining lattice homomorphism $\widehat{x^*}\in \fbl^{(p)}[E]$ (for any $1\leq p\leq \infty)$, there exists a lattice renorming $|||\cdot |||$ of $\fbl^{(p)}[E]$ in such a way that $\widehat{x^*}$ does not attain its norm.
\end{cor}
\begin{proof}
We may assume that $\|\widehat{x^*}\|=1$. Let $(x_n^*)_{n=1}^\infty\subset S_{E^*}$ be a sequence which converges to $x^*$ in $E^*$ such that $x_n^*$ is not proportional to $x^*$ for any $n\in\mathbb{N}$, and define $\mu:=\sum_{n=1}^\infty \frac{1}{2^n} \widehat{x_n^*}$. Note that since $\widehat{x_n^*} \perp \widehat{x^*}$ for all $n$, it follows that $\mu$ and $x^*$ are disjoint. On the other hand, $(x_n^*)_{n=1}^\infty$ converges to $x^*$ in the weak$^*$ topology of $B_{E^*}$, and elements of $\fbl^{(p)}[E]$ are weak$^*$-continuous on $B_{E^*}$. The proof is completed in a way similar to (iv)$\Rightarrow$(v) of Theorem \ref{thm:several}.
\end{proof}

It was conjectured in \cite[Conjecture 5.5]{DMCRARZ} that a functional $x^* \in E^*$ attains its norm if and only if the lattice homomorphism  $\widehat{x^*} \in \FBL[E]^*$ attains its norm. We have shown in Theorem \ref{THMAM} that every lattice homomorphism in $\FBL^{(\infty)}[E]^*$  attains its norm. Thus, the analogous conjecture turns out to be false in the setting of free $\infty$-convex Banach lattices generated by a Banach space. Nevertheless, we do not know the answer for free $p$-convex Banach lattices with $p \in [1,\infty)$:

\begin{question}
\label{QuestionFBL}
Let $E$ be a Banach space and $p \in [1,\infty)$. Does a functional $x^* \in E^*$ attain its norm if and only if the lattice homomorphism  $\widehat{x^*} \in \FBL^{(p)}[E]^*$ attains its norm?
\end{question}

In \cite[Definition 5.7]{DMCRARZ}, the authors introduced the following property: A Banach space $E$ has \textit{property (P)} if for every $x^*\notin \text{NA}(E,\mathbb{R})$, the set
$$
C:=\bigl\{y^*\in E^*\::\: |x^*(x)|+|y^*(x)|\leq \|x^*\| \text{ for every } x\in B_E\bigr\}
$$
satisfies that $x^*$ is in the $w^*$-closure of $\mathbb{R}^+C:=\{\lambda y^*\::\: \lambda>0,\:y^*\in C\}$. It is claimed in \cite[Lemma 5.8]{DMCRARZ} that Banach spaces with property (P) satisfy the above conjecture. However, there is a gap in the proof of this fact: at some point, the authors use the weak$^*$-continuity on bounded sets of the functions of $\fbl[E]$ and assume that this is sufficient to ensure that the elements of $\fbl[E]$ are weak$^*$-continuous on $\overline{\mathbb{R}^+ C}^{w^*}$. Since this set is not bounded in general, the argument is not valid. \medskip

For this reason, we propose the following slight modification and generalization of property (P): For $1\leq p<\infty$, a Banach space has \textit{property ($P_p$)} if for every $x^*\notin \text{NA}(E,\mathbb{R})$, the set
$$
C_p:=\bigl\{y^*\in E^*\::\: |x^*(x)|^p+|y^*(x)|^p\leq \|x^*\|^p \text{ for every } x\in B_E\bigr\}
$$
satisfies that $x^*$ is in $\bigcup_{n\in\mathbb{N}} \overline{(\mathbb{R}^+ C)\cap nB_{E^*}}^{w^*}$. It is clear that if $p\leq q$ then a Banach space with property ($P_p$), also has property ($P_q$).

\begin{prop}
Given $1\leq p<\infty$, let $E$ be a Banach space with property ($P_p$). Then, $x^*\in \text{NA}(E,\mathbb{R})$ if and only if $\widehat{x^*}\in\text{NA}(\fbl^{(p)}[E],\mathbb{R})$.
\end{prop}

\begin{proof}
The proof follows the same argument as in \cite[Lemma 5.8]{DMCRARZ}, despite the fact that there was a gap using the original definition of property ($P$).
\end{proof}

Fortunately, the classes of Banach spaces provided in \cite{DMCRARZ} satisfying property ($P$) also satisfy ($P_p$) for every $1\leq p<\infty$. More specifically, we have:
\begin{itemize}
    \item $E=c_0$ has property ($P_p$) with $x^*\in \overline{(\mathbb{R}^+ C_p)\cap 2 B_{E^*}}^{w^*}$ for every $x^*\notin \text{NA}(E,\mathbb{R})$.
    \item $E=\ell_1(\Gamma)$ has property ($P_p$) with $x^*\in \overline{(\mathbb{R}^+ C_p)\cap B_{E^*}}^{w^*}$ for every $x^*\notin \text{NA}(E,\mathbb{R})$.
    
\end{itemize}

For the first claim, see the proof of \cite[Theorem 5.10]{DMCRARZ}; the second claim is a particular case of the following result.

\begin{prop}
If $\mu$ is a localizable measure (i.e.~such that $L_1(\mu)^*=L_\infty(\mu)$), then $L_1(\mu)$ has property ($P_p$). 
\end{prop}
\begin{proof}
Let $x^*$ be a norm-one linear functional on $L_1(\mu)$ which does not attain its norm. Let us denote by $g$ the function in $B_{L_\infty(\mu)}$ such that $x^*(f)=\int f gd\mu$ for every $f\in L_1(\mu)$: we will identify $x^*=g$. First, note that $A=\{|g|=1\}$ is a $\mu$-null set: if not, take $B\subset A$ of finite measure and define $f=\frac{\text{sgn}(g)}{\mu(B)}\chi_B$, which is a norm-one function in $L_1(\mu)$ such that $x^*(f)=1$ and this would contradict the fact that $x^*\notin\text{NA}(L_1(\mu),\mathbb{R})$. Hence the decreasing sequence of measurable sets defined by $A_n=\{1\geq |g|>1-\frac{1}{n}\}$ has $\mu$-null intersection.\medskip

For every $n\in\mathbb{N}$, we define $y_n^*:=\frac{1}{n}g\cdot\chi_{A_n^c}$; note that $0\leq |g|\cdot\chi_{A_n^c}\leq |g|\in B_{L_\infty(\mu)}$, and so $ny_n^*\in B_{L_\infty(\mu)}$. We have
$$
|x^*|+|y_n^*|=|g|\chi_{A_n}+|g|\chi_{A_n^c}+\frac{1}{n}|g|\chi_{{A_n^c}}=|g|\chi_{A_n}+\left(1+\frac{1}{n}\right)|g|\chi_{A_n^c}\leq 1, \text{ $\mu$-a.e.}
$$
Thus, $(y_n^*)_{n=1}^\infty$ is a sequence in the set
$$
C_1=\bigl\{y^*\in L_1(\mu)^*\::\: |x^*(f)|+|y^*(f)|\leq 1 \text{ for every } f\in B_{L_1(\mu)}\bigr\}.
$$
Let us show that $\mathbb{R}^+ C_1\cap B_{L_1(\mu)^*}\ni ny^*_n\overset{w^*}{\to} x^*$. Indeed, for every $f\in L_1(\mu)$ we have that $fg\in L_1(\mu)$, and so

$$(x^*-ny_n^*)(f)=\int (g-g\cdot\chi_{A_n^c})fd\mu=\int_{A_n}gfd\mu\to 0,~ n\to\infty.$$
This shows that $L_1(\mu)$ has property $(P_1)$ and consequently ($P_p$) for every $1\leq p<\infty$.
\end{proof}

We now present an example of a Banach lattice $X$ such that $\text{Hom}(X,\mathbb{R})\subset \text{NA}(X,\mathbb{R})$, but it is neither order continuous nor an AM-space.

\begin{example}
Consider the free Banach lattice of a reflexive Banach space $E$ with $\dim E=\infty$. This space is not $\sigma$-order complete (see \cite[Proposition 2.11]{OTTT}), hence cannot be order continuous. Moreover, $\text{FBL}[E]$ is at most $2$-convex (see \cite[Proposition 9.30]{OTTT}), therefore cannot be $\infty$-convex, and so it is not lattice isomorphic to an AM-space.

Let $\widehat{T}:\fbl[E]\to Y$ be a compact lattice homomorphism. Note that $\widehat{T}$ is the unique lattice extension of the operator $T: E\to Y$ defined by $T:=\widehat{T}\,\delta_E$ and $\|T\|=\|\widehat{T}\|$. Since $E$ is reflexive and $T$ is compact, then $T$ attains its norm at some element $x\in B_E$, thus $\widehat{T}$ attains its norm at $\delta_x$. 
\end{example}

We end this section with an alternative proof of Theorem \ref{THMAM} which, chronologically, was the first one to be developed. It uses some advanced tools from free Banach lattices, as opposed to the more direct version presented in section \ref{sec:AMspaces}. However, we find it interesting enough to include it, as it shows how free Banach lattices can sometimes enlighten the way for solving general Banach lattice questions. We need first the following auxiliary lemma.

\begin{lem}
\label{LemmaFBL}
Let $1\leq p\leq \infty$ and let $X$ be a $p$-convex Banach lattice with $p$-convexity constant 1. If every lattice homomorphism of $\fbl^{(p)}[X]$ is norm-attaining, then so is every lattice homormorphism of $X$.
\end{lem}

\begin{proof}
Let $x^*\colon X\to\mathbb{R}$ be a lattice homomorphism such that $\|x^*\|=1$. If we denote by $\delta_X$ the canonical isometric embedding of $X$ into $\fbl^{(p)}[X]$, the universal property of $\fbl^{(p)}[X]$ ensures the existence of a unique lattice homomorphism $\beta\colon \fbl^{(p)}[X]\to X$ satisfying $\beta\delta_X=\text{id}_X$ and such that $\|\beta\|=1$. \medskip

If every lattice homomorphism of $\fbl^{(p)}[X]$ is norm-attaining, then, in particular, there exists $f\in B_{\fbl^{(p)}[X]}$ such that $\widehat{x^*}(f)=1$. Now, consider $x^*\circ \beta$, which is a lattice homomorphism of $\fbl^{(p)}[X]$. Since $\widehat{x^*}$ and $x^*\circ \beta$ are both lattice homomorphisms of $\fbl^{(p)}[X]$ which agree on $\delta_X(X)$, then, in fact, $\widehat{x^*}=x^*\circ\beta$. In particular,
$$
1=\widehat{x^*}(f)=(x^*\circ \beta)(f)=x^*(\beta f),
$$
and, as $\beta f\in B_X$, this shows that $x^*$ is norm-attaining.
\end{proof}

\begin{proof}[Alternative proof of Theorem \ref{THMAM}]
    To show that every lattice homomorphism on an AM-space attains its norm, we will check the following statements:
    \begin{enumerate}[(1)]
    \item Every lattice homomorphism of $\fbl^{(\infty)}[\ell_2]$ attains its norm.
    \item Given any separable Banach space $E$, every lattice homomorphism of $\fbl^{(\infty)}[\fbl^{(\infty)}[E]]$ attains its norm.
    \item Given any separable Banach space $E$, every lattice homomorphism of $\fbl^{(\infty)}[E]$ attains its norm.
    \item Given any Banach space $E$, every lattice homomorphism of $\fbl^{(\infty)}[E]$ attains its norm.
    \item Every lattice homomorphism of an AM-space attains its norm.
\end{enumerate}

\smallskip For $(1)$, observe that since $\ell_2$ is reflexive then, by James' theorem, all its functionals are norm-attaining. As the lattice homomorphisms of $\fbl^{(\infty)}[\ell_2]$ are of the form $\widehat{x^*}$, for some $x^*\in \ell_2$ (and are defined by $\widehat{x^*}(f)=f(x^*)$), this shows that every lattice homomorphism of $\fbl^{(\infty)}[\ell_2]$ attains its norm. \smallskip

$(2)$: If $E$ is a separable Banach space, then $\fbl^{(\infty)}[E]$ is a separable AM-space. By \cite[Theorem 2]{LL}, $\fbl^{(\infty)}[E]$ has a monotone Schauder basis, so by \cite[Theorem 10.24]{OTTT} $\fbl^{(\infty)}[\fbl^{(\infty)}[E]]$ is lattice isometric to $\fbl^{(\infty)}[\ell_2]$ and the conclusion follows from $(1)$. \smallskip

$(3)$: Since  $\fbl^{(\infty)}[E]$ is an AM-space ($\infty$-convex with constant 1), this implication is an immediate consequence of Lemma \ref{LemmaFBL} and $(2)$. \smallskip

$(4)$: Let $E$ be an arbitrary Banach space. Let $\widehat{x^*}\in\text{Hom}(\text{FBL}[E],\mathbb{R})$ such that $\|\widehat{x^*}\|=\|x^*\|=1$. There is a sequence $(x_n)_n\subset B_E$ such that $x^*(x_n)\rightarrow 1$ as $n\to\infty$. We define
$$
F:=\overline{\text{span}}\bigl\{x_n\::\:n\in\mathbb{N}\bigr\},
$$
which is a separable subset of $E$. By \cite[Proposition 3.10]{OTTT}, the mapping $\overline{\iota}:\fbl^{(\infty)}[F]\to \fbl^{(\infty)}[E]$ defined by $\overline{\iota}f=f\circ \iota^*$ for every $f\in\fbl^{(\infty)}[F]$ (where $\iota:F\hookrightarrow E$ is the canonical inclusion of $F$ into $E$ as a subspace) is a lattice isometric embedding. \smallskip

Now observe that $y^*=\left.\widehat{x^*}\right|_{\overline{\iota}(\fbl^{(\infty)}[F])}$ is a
norm-one lattice homomorphism of $\overline{\iota}(\fbl^{(\infty)}[F])$. Indeed, for every $n\in \mathbb N$ we have
$$
y^*(\overline{\iota}\delta_{x_n})=\widehat{x^*}(\overline{\iota}\delta_{x_n})=\widehat{x^*}(\delta_{x_n}\circ \iota^*)=(\delta_{x_n}\circ i^*)(x^*)=\delta_{x_n}(\iota^*x^*)=\iota^*x^*(x_n)=x^*(\iota x_n)=x^*(x_n).
$$
Since $F$ is separable, we know by $(3)$ that every lattice homomorphism of $\fbl^{(\infty)}[F]$ attains its norm, and as $\overline{\iota}$ is a lattice isometry, $y^*$ is norm-attaining. That is, there is $f\in \overline{\iota}(\fbl^{(\infty)}[F])\subset \fbl^{(\infty)}[E]$, $\|f\|_\infty=1$ such that $y^*(f)=1$. Thus, $\widehat{x^*}$ also attains its norm at $f$. \smallskip

$(5)$ follows now directly from the above and the preceding lemma.
\end{proof}

\section{Free Banach lattices over lattices}
\label{Section:FBLlattices}

In this section we focus on the setting of free Banach lattices generated by lattices. Here, a \textit{lattice} is a non-empty set $\mathbb{L}$ with a partial order such that for every $x,y\in\mathbb{L}$, the set $\{x,y\}$ has both a supremum $x\lor y$ and an infimum $x\land y$. Throughout this section, by \textit{lattice homomorphism} we refer to any map $T:\mathbb{L}\rightarrow \mathbb{M}$ between two lattices $\mathbb{L}$ and $\mathbb{M}$ that preserves lattice operations, i.e., $T(x\vee y)=Tx \vee Ty$ and $T(x\wedge y)=Tx \wedge Ty$ for every $x,y\in \mathbb{L}$. To avoid confusion, we will use the term \textit{linear lattice homomorphism} when referring to linear and bounded operators between Banach lattices that preserve lattice operations. \medskip

The free Banach lattice generated by a lattice $\mathbb{L}$, denoted $\fbl\langle\mathbb{L}\rangle$, was introduced in \cite{AR-A19}, following a similar approach to \cite{ART18}. Later, it was deduced in \cite[Theorem 3.9]{AMRR22} that $\fbl\langle\mathbb{L}\rangle$ is always 2-isomorphic to an AM-space. In this section, we provide a direct proof of this fact, and furthermore we show that every linear lattice homomorphism on the free Banach lattice generated by a lattice attains its norm. But first, we extend the definition of the free Banach lattice generated by a lattice to the setting of $p$-convex Banach lattices for $1\leq p\leq \infty$, in a analogous way to \cite{JLTTT}.

\begin{defn}
Given $1\leq p\leq \infty$, the \textit{free $p$-convex Banach lattice} over a lattice $\mathbb{L}$ is a $p$-convex Banach lattice $\fbl^{(p)}\langle\mathbb{L}\rangle$ (with $p$-convexity constant 1) together with a norm-bounded lattice homomorphism $\phi:\mathbb{L}\to \fbl^{(p)}\langle\mathbb{L}\rangle$ with the property that for every $p$-convex Banach lattice $X$ (with $p$-convexity constant 1) and every norm-bounded lattice homomorphism $T:\mathbb{L}\to X$ there is a unique linear lattice homomorphism $\widehat{T}:\fbl^{(p)}\langle\mathbb{L}\rangle\to X$ such that $T=\widehat{T}\circ\phi$ and $\|\widehat{T}\|=\|T\|$. Here ``norm-bounded'' means $\|T\|:=\sup\{\|T(x)\|\::\:x\in\mathbb{L}\}<\infty$.
\end{defn}

As in the case of free Banach lattices generated by Banach spaces, when $p=1$ we recover the free Banach lattice generated by a lattice, $\fbl\langle\mathbb{L}\rangle$, while $\fbl^{(\infty)}\langle\mathbb{L}\rangle$ can also be called the \textit{free AM-space over} $\mathbb{L}$, since the classes of AM-spaces and $\infty$-convex Banach lattices with constant $1$ coincide.

\begin{rem}
    In the definition above, the assumption of norm-boundedness of $T$ followed by the requirement $\|\widehat{T}\|=\|T\|$ can be replaced with the assumption that $\|T\|\leq 1$ and the requirement that $\|\widehat{T}\|\leq 1$.
\end{rem}

The existence of $\fbl\langle \mathbb{L}\rangle$ was established in \cite[Section 2]{AR-A19}. More specifically, it is shown that $\fbl\langle \mathbb{L}\rangle$ can be identified with the quotient of $\fbl[\ell_1(\mathbb{L})]$ with respect to the closed ideal $I$ generated by the set 
$$\{u(x)\vee u(y)-u(x\vee y), \,u(x)\wedge u(y)-u(x\wedge y):x,y\in\mathbb{L}\},$$
where $u:\mathbb{L}\rightarrow \fbl[\ell_1(\mathbb{L})]$ denotes the canonical embedding (recall that for any set $A$, $\fbl[\ell_1(A)]$ coincides with $\fbl(A)$, the free Banach lattice generated by the set $A$, introduced in \cite{dPW15}). Arguing analogously, it can be shown that $\fbl^{(p)}\langle\mathbb{L}\rangle$ exists for every $1\leq p\leq \infty$ and coincides with the quotient of $\fbl^{(p)}[\ell_1(\mathbb{L})]$ with respect to the closed ideal $I$ generated by the same set above.\medskip

In order to investigate these objects more thoroughly, let us start by recalling and setting the notation we will use. We will write:
$$
\mathbb{L}^*=\{x^*:\mathbb{L}\to [-1,1] \::\: x^* \text{ is a lattice homomorphism}\}.
$$
By Tychonoff's theorem, $[-1,1]^{\mathbb{L}}$ is a compact space with respect to the product topology and it is not difficult to check that $\mathbb{L}^*$ is a positively balanced closed subset of it, so $\mathbb{L}^*$ is a compact space with the product topology. We will always assume that $\mathbb{L}^*$ is equipped with this topology. Let $C_{ph}(\mathbb{L}^*)$ stand for the space of continuous positively homogeneous functions on $\mathbb{L}^*$ endowed with the supremum norm $\|\cdot\|_{\infty}$. For every $x\in \mathbb{L}$, we consider the evaluation function $\delta_x:\mathbb{L}^*\to\mathbb{R}$ given by $\delta_x(x^*)=x^*(x)$. It is easy to see that $\delta_x\in C_{ph}(\mathbb{L}^*)$, for every $x\in \mathbb{L}$.

\begin{prop}\label{prop: freeam}
$C_{ph}(\mathbb{L}^*)= \overline{\lat}^{\|\cdot\|_\infty}\{\delta_x:x\in\mathbb{L}\}$ together with the map $\phi_\infty(x)=\delta_x$ is the free AM-space over $\mathbb{L}$.
\end{prop}
\begin{proof}
Let us first prove the equality. According to part (i) of Proposition \ref{prop:evaluation} it is enough to show that if $x^*$ and $y^*$ in $\mathbb{L}^*\backslash\{0\}$ are not positive multiples, there is $f\in \text{lat}\{\delta_x,~x\in \mathbb{L}\}$ such that $f(x^*)=0\ne f(y^*)$. Indeed, suppose that $x^*=\lambda y^*$ for some $\lambda<0$. We may assume that there exists $x\in\mathbb{L}$ such that $y^*(x)>0$, so if we define $f:=\delta_x^+=\delta_x\lor 0$, then $f(x^*)=0< f(y^*)$. On the other hand, if $x^*$ and $y^*$ are not proportional, then there are $u,v\in \mathbb{L}$ such that $(x^*(u), x^*(v))$ and $(y^*(u), y^*(v))$ are not proportional.
Without loss of generality we suppose $x^*(v)\neq 0$. 
Take $f:=\delta_u-\frac{x^*(u)}{x^*(v)}\delta_v$. It is clear that $f(x^*)=0$. On the other hand, $f(y^*)=y^*(u)-\frac{x^*(u)}{x^*(v)}y^*(v)\ne 0$, as $(x^*(u), x^*(v))$ and $(y^*(u), y^*(v))$ are not proportional.\medskip

Now let $X$ be an AM-space and $T:\mathbb{L}\to X$ a lattice homomorphism with range contained in $B_X$. By Proposition \ref{prop:norming} we may assume that $X=C_{ph}(K_X)$. Every $z^*\in K_X$ maps $B_X$ into $[-1,1]$, and so $\psi(z^*):=z^*\circ T$ is a lattice homomorphism from $\mathbb{L}$ into $[-1,1]$, i.e. $\psi(z^*)\in \mathbb{L}^*$. It is easy to see that $\psi:K_X\to \mathbb{L}^*$ is positively homogeneous and continuous with respect to the weak* topology on $K_X$ and the product topology on $\mathbb{L}^*$. Hence, the composition operator $\widehat{T}: C_{ph}(\mathbb{L}^*)\to C_{ph}(K_X)$ given by $\widehat{T}f(z^*)=f(\psi(z^*))$ is a well-defined linear lattice homomorphism of norm at most $1$. Finally, for every $x\in \mathbb{L}$ and $z^*\in K_X$ we have $$\widehat{T}\delta_x(z^*)=\delta_x(\psi(z^*))=\psi(z^*)(x)=z^*(Tx),$$ hence $\widehat{T}\delta_x=Tx$.

For the uniqueness of extension assume that $S:C_{ph}(\mathbb{L}^*)\to X$ is a linear lattice homomorphism which agrees with $\widehat{T}$ on $\delta_x$, for every $x\in \mathbb{L}$, then they have to agree on the closed sublattice generated by these vectors. As $\overline{\lat}^{\|\cdot\|_\infty}\{\delta_x:x\in\mathbb{L}\}=C_{ph}(\mathbb{L}^*)$, we conclude that $\widehat{T}=S$.
\end{proof}

Let us now switch to the case $1\leq p<\infty$. For $f\in\mathbb{R}^{\mathbb{L}^*}$, define
\begin{equation}\label{eq:free-norm-lattice}
 \|f\|_p=\sup\left\{\left(\sum_{i=1}^n|f(x_i^*)|^p\right)^\frac{1}{p} \::\:n\in\mathbb{N},\:x_1^*,\ldots,x_n^*\in\mathbb{L}^*,\:\sup_{x\in\mathbb{L}}\sum_{i=1}^n|x_i^*(x)|^p \leq 1\right\}.
\end{equation}
It was proven in \cite[Theorem 1.2]{AR-A19} that $$\fbl\langle \mathbb{L}\rangle=\overline{\lat}^{\|\cdot\|_1}\{\delta_x:x\in\mathbb{L}\},$$ 
where the closure is taken in $\{f\in\mathbb{R}^{\mathbb{L}^*}\::\:\|f\|_1<\infty\}$, and $\phi(x)=\delta_x$, for every $x\in\mathbb{L}$. It is not difficult to check that the same proof can be adapted verbatim to show that 
$$\fbl^{(p)}\langle \mathbb{L}\rangle=\overline{\lat}^{\|\cdot\|_p}\{\delta_x:x\in\mathbb{L}\},$$
where this time the closure is taken in $\{f\in\mathbb{R}^{\mathbb{L}^*}\::\:\|f\|_p<\infty\}$.\medskip

Clearly, $\|\cdot\|_\infty\leq \|\cdot\|_p$ for every $1\leq p<\infty$. Since $\lat\{\delta_x:x\in\mathbb{L}\}$ is contained in $C_{ph}(\mathbb{L}^*)$, it follows that the identity operator $id_p:\fbl^{(p)}\langle \mathbb{L}\rangle \rightarrow C_{ph}(\mathbb{L}^*)$ is bounded. On the other hand, the universal property of $\fbl\langle \mathbb{L}\rangle$ implies that there exists a contractive linear lattice homomorphism $\widehat{\phi_p}:\fbl\langle \mathbb{L}\rangle\rightarrow \fbl^{(p)}\langle \mathbb{L}\rangle$ that extends the canonical map $\phi_p:\mathbb{L}\rightarrow \fbl^{(p)}\langle \mathbb{L}\rangle$, and hence, coincides with the identity map on $\lat\{\delta_x:x\in\mathbb{L}\}$. It is clear that $id_p\circ \widehat{\phi_p}=id_1$, so actually $\widehat{\phi_p}$ must be the identity from $\fbl\langle \mathbb{L}\rangle$ into $\fbl^{(p)}\langle \mathbb{L}\rangle$, and we can properly say that $\|f\|_p\leq \|f\|_1$ for every $f\in \fbl\langle \mathbb{L}\rangle$. Additionally, it can be deduced from \cite[Theorem 3.6]{AMRR22} that $\|\cdot\|_1\leq 2\|\cdot \|_\infty$ (where the authors show that $\fbl\langle\mathbb{L}\rangle$ is $2$-linearly lattice isomorphic to an AM-space for any $\mathbb{L}$). We include here a direct proof of this fact.

\begin{prop}\label{prop:fbll}
Let $1\leq p<\infty$. For every $f\in C_{ph}(\mathbb{L}^*)$ we have $\|f\|_p\leq 2^\frac{1}{p}\|f\|_\infty$. Moreover, $\left(C_{ph}(\mathbb{L}^*),\|\cdot\|_p\right)$ together with the map $\phi_p(x)=\delta_x$ is the free $p$-convex Banach lattice over $\mathbb{L}$.
\end{prop}
\begin{proof}
In order to prove the first claim, it is enough to show that if $x_1^*,\ldots,x_n^*\in\mathbb{L}^*$ are such that $\sup_{x\in\mathbb{L}}\sum_{i=1}^n |x_i^*(x)|^p\leq 1$, then $\left(\sum_{i=1}^n |f(x_i^*)|^p\right)^\frac{1}{p}\leq 2^\frac{1}{p}\|f\|_{\infty}$. 

Fix $\varepsilon>0$. For each $i=1,\ldots,n$, we define $m_i=\inf\{x_i^*(x)\::\: x\in\mathbb{L}\}$ and $M_i=\sup\{x_i^*(x)\::\: x\in\mathbb{L}\}$ and find $x_{i,m},\,x_{i,M}\in\mathbb{L}$ such that
$$
0\leq x_i^*(x_{i,m})-m_i\leq \varepsilon \quad \text{ and } \quad 0\leq M_i-x_i^*(x_{i,M})\leq \varepsilon.
$$
Now, define $x_m=\bigwedge_{i=1}^n x_{i,m}$ and $x_M=\bigvee_{i=1}^n x_{i,M}$. For every $1\leq i\leq n$ we have $$
0\leq x_i^*(x_{m})-m_i\leq \varepsilon \quad \text{ and } \quad 0\leq M_i-x_i^*(x_{M})\leq \varepsilon.
$$
It follows that $$1\geq \left(\sum_{i=1}^n |x_i^*(x_m)|^p\right)^\frac{1}{p}\geq \left(\sum_{i=1}^n |m_i|^p\right)^\frac{1}{p}-n^\frac{1}{p}\varepsilon,$$ and 
$$1\geq \left(\sum_{i=1}^n |x_i^*(x_M)|^p\right)^\frac{1}{p}\geq \left(\sum_{i=1}^n |M_i|^p\right)^\frac{1}{p}-n^\frac{1}{p}\varepsilon.$$ For every $1\leq i\leq n$ we have that $y_i^*:=\frac{1}{|m_i|\vee |M_i|}x_i^*\in \mathbb{L}^*$, hence $$|f(x^*_i)|=(|m_i|\vee |M_i|)|f(y^*_i)|\leq (|m_i|\vee |M_i|)\|f\|_{\infty}.$$ We conclude that 
\begin{align*}
    \left(\sum_{i=1}^n |f(x_i^*)|^p\right)^\frac{1}{p} & \leq \|f\|_{\infty}\left(\sum_{i=1}^n (|m_i|\vee |M_i|)^p\right)^\frac{1}{p}\\
    & \leq \|f\|_{\infty}\left(\sum_{i=1}^n |m_i|^p+\sum_{i=1}^n |M_i|^p\right)^\frac{1}{p}\leq 2^\frac{1}{p}(1+n^\frac{1}{p}\varepsilon)\|f\|_{\infty}.
\end{align*}
As $\varepsilon$ was arbitrary the claim follows.\medskip

Note that $C_{ph}(\mathbb{L}^*)$ is a Banach lattice with respect to $\|\cdot\|_p$, which contains $\{\delta_x:x\in\mathbb{L}\}$. Hence, according to the discussion before the proposition, in order to prove the second claim, it is enough to show that $$C_{ph}(\mathbb{L}^*)=\overline{\lat}^{\|\cdot\|_p}\{\delta_x:x\in\mathbb{L}\}.$$
Since the two norms are equivalent, the closure may be taken with respect to $\|\cdot\|_\infty$, and so the statement follows from Proposition \ref{prop: freeam}.
\end{proof}

We conclude this section by showing that every linear lattice homomorphism on $\fbl^{(p)}\langle\mathbb{L}\rangle$ attains its norm. Note that although we have just seen that the free norm of $\fbl^{(p)}\langle \mathbb{L}\rangle$ is equivalent to the supremum norm, this does not guarantee that its linear lattice homomorphisms attain their norm. Also observe that it is not difficult to find very simple examples of lattice homomorphisms $x^*:\mathbb{L}\to\mathbb{R}$ which do not attain their norm: the naturals $\mathbb{N}$ with the usual order and $x^*:\mathbb{N}\to\mathbb{R}$ defined by $x^*(n):=1-\frac{1}{n}$; the interval $(-1,1)$ with the usual order and $x^*$ being the identity on $(-1,1)$; another example is, given an AM-space $X$, to consider $\mathbb{L}=\{x\in X\::\:\|x\|<1\}$ and take as $x^*$ the restriction to $\mathbb{L}$ of a norm-one linear lattice homomorphism defined on $X$.

\begin{prop}
\label{prop:normattainingFBLlattice}
Every linear lattice homomorphism on $\fbl^{(p)}\langle\mathbb{L}\rangle$ attains its norm.
\end{prop}
\begin{proof}
According to Proposition \ref{prop:fbll} we may assume that $\fbl^{(p)}\langle\mathbb{L}\rangle =C_{ph}(\mathbb{L}^*)$ endowed with the norm $\|\cdot\|_p$. It follows from Proposition \ref{prop:evaluation} that we need to show that point evaluations at the members of $\mathbb{L}^*$ attain their norms. Let $0\ne x^*\in \mathbb{L}^*$ and let $m=\inf\{x^*(x)\::\: x\in\mathbb{L}\}$ and $M=\sup\{x^*(x)\::\: x\in\mathbb{L}\}$. We will assume that $|M|\geq |m|$, which also implies that $M>0$ (the case $|M|\leq |m|$ is analogous). Moreover, replacing $x^*$ with $\frac{1}{M}x^*$ if needed, we may assume that $M=1$. \medskip

Let us first show that $\|\widehat{x^*}\|\leq 1$, where $\widehat{x^*}$ denotes the point evaluation at $x^*$, or equivalently, the unique extension of the lattice homomorphism $x^*$ to $\fbl^{(p)} \langle\mathbb{L}\rangle$ as a linear lattice homomorphism. Indeed, if $f\in C_{ph}(\mathbb{L}^*)$ is such that $\|f\|_p\leq 1$, then $\|f\|_{\infty}\leq 1$, and so $\widehat{x^*}(f)=f(x^*)\leq 1$. On the other hand, for every $n\in\N$ there is $x_n\in\mathbb{L}$ such that $\widehat{x^*}(\delta_{x_n})=x^*(x_n)\geq 1-\frac{1}{n}$. By replacing $x_n$ by $\lor_{k=1}^n x_k$ if necessary, we may assume that $(x_n)_{n=1}^\infty$ is an increasing sequence in $\mathbb{L}$. Since $\phi_p:\mathbb{L}\rightarrow\fbl^{(p)}\langle\mathbb{L}\rangle$ preserves lattice operations, $(\delta_{x_n}^+)_{n=1}^\infty$ defines an increasing sequence in $B_{\fbl^{(p)}\langle\mathbb{L}\rangle^+}$. Thus, by Proposition \ref{prop:characterizationLatNA}, $\widehat{x^*}$ attains its norm.
\end{proof}

\section*{Acknowledgments}
We wish to thank A. Avil\'es and M. A. Taylor for helpful discussions concerning the topic of this paper. 

E.~Bilokopytov is supported by Pacific Institute for the Mathematical Sciences. E. Garc\'ia-S\'anchez was partially supported by FPI grant CEX2019-000904-S-21-3 funded by MCIN/AEI/10.13039/501100011033 and by ``ESF+''. D.~de Hevia benefited from an FPU Grant FPU20/03334 from the Ministerio de Universidades. Research of E. Garc\'ia-S\'anchez, D. de Hevia and P. Tradacete has been partially supported by grants PID2020-116398GB-I00 and CEX2019-000904-S funded by MCIN/AEI/10.13039/501100011033. G. Mart\'inez-Cervantes is partially supported by grants PID2021-122126NB-C32 fun\-ded by MCIN/AEI/10.13039/501100011033, by ERDF A way of making Europe, and by Fundaci\'{o}n S\'{e}neca - ACyT Regi\'{o}n de Murcia (21955/PI/22). P.~Tradacete was also partially supported by grant PCI2024-155094-2 funded by MCIN/AEI/10.13039/501100011033.

\end{document}